\documentclass[letterpaper,12pt]{article}

\usepackage{tabularx} 
\usepackage{amsmath,amsfonts,stmaryrd,amssymb,bm,mathrsfs, dsfont}  
\usepackage{graphicx} 
\usepackage{subcaption}
\usepackage{natbib}
\usepackage{algorithm,algorithmic}
\usepackage{booktabs}
\usepackage{authblk}
\usepackage{multirow}
\usepackage{paralist}
\usepackage{amsthm} 
\usepackage[margin=1in,letterpaper]{geometry} 
\usepackage[final]{hyperref} 
\usepackage{verbatim}
\hypersetup{
	colorlinks=true,       
	linkcolor=cyan,        
	citecolor=cyan,        
	filecolor=magenta,     
	urlcolor=cyan         
}
\usepackage{blindtext}
\allowdisplaybreaks[4]
 
\newtheorem{theorem}{Theorem}
\newtheorem{lemma}{Lemma}
\newtheorem{assump}{Assumption}
\newtheorem{remark}{Remark}

\DeclareMathOperator*{\argmin}{arg\,min}
\DeclareMathOperator*{\arginf}{arg\,inf}

\begin{document}
\title{Limit Results for Estimation of Connectivity Matrix in Multi-layer Stochastic Block Models}

\author[1]{\large Wenqing Su}
\author[2]{\large \;Xiao Guo}
\author[1]{\large Ying Yang}
\affil[1]{\normalsize Department of Mathematical Sciences, Tsinghua University}
\affil[2]{\normalsize School of Mathematics, Northwest University}

\date{}  

\maketitle

\begin{abstract}
Multi-layer networks arise naturally in various domains including biology, finance and sociology, among others. The multi-layer stochastic block model (multi-layer SBM) is commonly used for community detection in the multi-layer networks. Most of current literature focuses on statistical consistency of community detection methods under multi-layer SBMs. However, the asymptotic distributional properties are also indispensable which play an important role in statistical inference. In this work, we aim to study the estimation and asymptotic properties of the layer-wise scaled connectivity matrices in the multi-layer SBMs. We develop a novel and efficient method to estimate the scaled connectivity matrices. Under the multi-layer SBM and its variant multi-layer degree-corrected SBM, we establish the asymptotic normality of the estimated matrices under mild conditions, which can be used for interval estimation and hypothesis testing. Simulations show the superior performance of proposed method over existing methods in two considered statistical inference tasks. We also apply the method to a real dataset and obtain interpretable results. 
\end{abstract}

\section{Introduction}
The multi-layer network has gained widespread attention for its ability to represent multiple relationships among the same entities of interest, thereby enhancing the understanding of complex network data \citep{mucha2010community, holme2012temporal, kivela2014multilayer, han2015consistent}. For example, the Worldwide Food and Agricultural Trade (WFAT) multi-layer trade network \citep{de2015structural}, encompassing trade relationships among the same countries (nodes) across different commodities (layers), offers a richer insight than a single-layer trade network. 

The multi-layer stochastic block model (multi-layer SBM) \citep{han2015consistent, paul2016consistent} has been popularly used for the community detection task in the multi-layer network, where each layer corresponds to a stochastic block model (SBM) \citep{holland1983stochastic}. In particular, given common unobserved communities (blocks), the edges in each layer are generated according to a layer-wise block probability matrix or the so-called connectivity matrix. Various methods have been proposed to detect the communities of multi-layer networks under the multi-layer SBM, see, e.g., \cite{han2015consistent, bhattacharyya2018spectral, paul2020spectral, jing2021community, arroyo2021inference, LeiJ2022Bias}. 

Most of the aforementioned literature on community detection of multi-layer SBMs is devoted to the statistical consistency. However, the asymptotic distributional properties are also vital and particularly useful in subsequent inference tasks. In the context of single-layer SBMs or other related models, various versions of asymptotically Gaussian behavior emerged for different purposes \citep{agterberg2023overview}. Specifically, \cite{tang2018limit} provided a central limit theorem for the eigenvectors of the normalized Laplacian matrix under the random dot product graph (RDPG) model \citep{young2007random}. \cite{rubin2022statistical} extended these asymptotic results to the generalized RDPG and applied the theory to SBMs. \cite{bickel2013asymptotic} and \cite{tang2022asymptotically} established the asymptotic normality results for the estimation of the connectivity matrix using the maximum likelihood method and spectral embedding, respectively. \cite{fan2022simple} tested whether two nodes from a mixed-membership model \citep{airoldi2008mixed} have the same membership parameters by establishing the asymptotic distributional properties. \cite{jin2022optimal} proposed a goodness-of-fit approach to select the number of communities and the asymptotic null distribution of the test statistic is normal. Also see \cite{agterberg2023overview} for a comprehensive overview of asymptotic normality in SBMs.

In the multi-layer SBM, the asymptotic distributional properties are crucial for subsequent inferences especially for those tailored for the multi-layer networks. However, as far as we are aware, the asymptotic properties under multi-layer SBMs are largely unexplored. An important inference task in multi-layer SBMs is to test whether networks across different layers are generated from the same SBM, essentially the same connectivity matrix given the consensus community memberships. For instance, in analyzing the WFAT multi-layer network, the goal  extends beyond the community detection to ascertaining whether trading patterns across different commodities are similar. Asymptotic properties play a fundamental role in such inference. In the literature, \cite{arroyo2021inference} and \cite{zheng2022limit} showed that the connectivity matrices estimated using the spectral embedding are asymptotically normal under the common subspace independent edge graph (COSIE) model. These studies necessitate that the connectivity matrices for \emph{all} network layers are full rank. However, in practical multi-layer networks like WFAT, not all layers carry comprehensive cluster information, potentially leading to missing clusters, which results in the  rank-deficiency of individual connectivity matrices \citep{su2023}. Moreover, the methods in \cite{arroyo2021inference} and \cite{zheng2022limit} treat multi-layer SBMs as special case of their COSIE model, rather than being specifically tailored for them. 


Motivated by the above problems, we study the asymptotic properties in the estimation of connectivity matrices under multi-layer SBMs, without imposing the full rank assumption on their populations. We develop a simple and efficient {s}pectral {c}lustering-based method for the \emph{{scaled} {c}onnectivity matrix} {e}stimation, where the scaled connectivity matrix is also called the score matrix \citep{arroyo2021inference}.
Under proper conditions, we establish the asymptotic normality of the estimated scaled connectivity matrices under multi-layer SBMs and its variants, the degree-corrected multi-layer SBM. We emphasize that the asymptotic properties can be used to perform various statistical inferences, such as interval estimation and hypothesis testing.

The main contributions of this paper are threefold. First, we systematically study and specifically tailor the asymptotic normality of the scaled connectivity matrix estimation under the multi-layer SBM and its variant, the multi-layer degree-corrected SBM. To the best of our knowledge, the asymptotic normality under the degree-corrected SBM is less explored, let alone the multi-layer degree-corrected SBM. Second, the conditions under which our results hold are relatively mild. Compared to previous work which necessitated full rank in each layer \citep{arroyo2021inference, zheng2022limit}, we allow individual connectivity matrices to be rank-deficient and only require their squared summation to be full rank. This adaptability makes our approach particularly suitable for analyzing multi-layer networks where each layer only captures partial underlying communities. Third, we apply the established asymptotic normality to two statistical inference tasks, one for interval estimation of the scaled connectivity matrices, and the other for testing whether the population matrices corresponding to different layers are the same. The numerical results show superior performance of the proposed method over compared methods in these tasks.  The analysis of a real-world trade network dataset provides interpretable results.

The rest of the paper is organized as follows. Section \ref{MSBM} introduces the estimation method of individual connectivity matrices in the multi-layer SBM. Section \ref{asymptotic} provides the asymptotic normality of the estimates. Section \ref{MDCSBM} extends the method to multi-layer degree-corrected SBMs and presents the corresponding asymptotic normality results. Sections \ref{simulations} and \ref{Real data analysis} include the simulation and real data experiment. Section \ref{conclusion} concludes the paper and provides possible extensions. Technical proofs are included in the Appendix.

\section{Methodology}\label{MSBM}
In this section, we first present and reparameterize the multi-layer stochastic block model. Then we propose the method for estimating the individual connectivity matrices. The following notation is used throughout the paper. 

\noindent\emph{Notation:} We use $[n]$ to denote the set $\{1,...,n\}$. For a matrix $A\in \mathbb R^{n\times n}$, $\|A\|_F$ and $\|A\|_{1, \infty}$ denote the Frobenius norm and the maximum row-wise $\ell_1$ norm, respectively. The spectral norm of a matrix and the Euclidean norm of a vector are denoted by $\|\cdot\|_2$. The notation ${\rm diag}(\cdot)$ is used to denote a diagonal matrix formed from either a vector or the diagonal elements of a matrix. The notation ${\rm vec}(\cdot)$ refers to the vectorization of the upper triangular part of a symmetric matrix, proceeding in a column-wise order. For two sequences $a_n$ and $b_n$, we write $a_n \lesssim b_n$ if there exists some constant $c>0$ such that $a_n \leq cb_n$, and use $a_n \asymp b_n$ to denote the case where both $a_n \lesssim b_n$ and  $b_n \lesssim a_n$. The notation $a_n = \omega(b_n)$ denotes that $b_n / a_n \rightarrow 0$ as $n$ goes to infinity, while $a_n = o(b_n)$ means that $b_n = \omega(a_n)$.

\subsection{Multi-layer SBMs}
Consider the multi-layer network comprising $L$ layers and $n$ shared nodes, with its symmetric adjacency matrices represented by $A_l$, where $A_l \in \{0, 1\}^{n\times n}$ for all $1 \leq l \leq L$. The multi-layer SBM \citep{han2015consistent, valles2016multilayer, paul2016consistent} is as follows. 

Assume the $n$ nodes are partitioned into $K$ layer-independent communities with the community assignment of node $i$ denoted by $g_i\in[K]$. With this community assignment, the element $A_{l,ij}(i<j)$ of $A_l(l\in [L])$ is generated independently as the following SBM \citep{holland1983stochastic},
\begin{equation*}\label{adj_generate}
	A_{l,ij}\sim {\rm Bernoulli}(\rho B_{l,g_ig_j}),
\end{equation*}
where $\rho\in(0,1]$ controls the overall sparsity of the generated networks, $B_l\in \mathbb [0,1]^{K\times K}$ denotes the \emph{heterogeneous} connectivity matrix, and we assume the diagonal entries $A_{l,ii}$'s are all 0. 

It is easy to see that $$Q_l=\rho \Theta B_l \Theta^T \in [0,1]^{n \times n}$$ is the population matrix for 
$A_l$, in the sense that $Q_l-{\rm diag}(Q_l)=\mathbb E(A_l)$, where $\Theta\in \{0,1\}^{n\times K}$ denotes the membership matrix with $\Theta_{ik}=1$ if and only if $k=g_i$ and $\Theta_{ik}=0$ otherwise. Define $\Delta = {\rm diag}(n_1, \ldots, n_K)$, where each diagonal element $n_k$ represents the number of nodes in community $k\in[K]$, we can then rearrange $Q_l$ as
\begin{equation}\label{Q_l}
	Q_l = \rho\Theta \Delta^{-1/2} \Delta^{1/2} B_l \Delta^{1/2} \Delta^{-1/2} \Theta^T := UM_lU^T,
\end{equation}
where $U = \Theta \Delta^{-1/2}$ and $M_l = \rho \Delta^{1/2} B_l \Delta^{1/2}$. Here, $U$ is an $n \times K$ orthogonal column matrix with $K$ different rows and serves as the eigenspace of the population matrix $Q_l$. The $K\times K$ symmetric matrix $M_l$ is referred to as the scaled connectivity matrix or the score matrix. We will study its estimation and asymptotic normality in subsequent sections. Note that \cite{arroyo2021inference} first considered this type of decomposition when they developed the COSIE model. 

Through the decomposition in (\ref{Q_l}), the multi-layer SBM can be reparameterized by ($U, M_1, \ldots, M_L$). Note that such parametrization is identifiable up to orthogonal transformation because for any orthogonal matrix $Z \in \mathbb O(K)$, $(UZ) Z^TM_lZ (UZ)^T$ provides another valid decomposition, where $\mathbb O(K)$ represents the set consists of all $K\times K$ orthogonal matrices.

\subsection{Estimation of the scaled connectivity matrices}
In order to estimate the scaled connectivity matrices $M_l (l\in [L])$, we first estimate the common eigenspace $U$. 

To utilize the information across layers, we regard $U$ as the eigenspace of the following $\sum_{l=1}^LQ_l^2$,
\begin{equation*}
	\sum_{l=1}^LQ_l^2 = \sum_{l=1}^L\rho^2\Theta B_l \Theta^T \Theta B_l \Theta^T = U\sum_{l=1}^L\rho^2\Delta^{1/2} B_l \Delta B_l \Delta^{1/2} U^T,
\end{equation*}
where we squared $Q_l$'s before summing them up to avoid the community cancellation by direct summation \citep{LeiJ2022Bias}. Suppose $\sum_{l=1}^LB_l^2$ is of full rank, and recalling that $\Delta$ is a full rank diagonal matrix, then $\sum_{l=1}^L \rho^2\Delta^{1/2} B_l \Delta B_l \Delta^{1/2}$ is also of full rank.  Denote the $\sum_{l=1}^L \rho^2\Delta^{1/2} B_l \Delta B_l \Delta^{1/2}$ by $W \Lambda W^T$, where $W\in \mathbb O(K)$ and $\Lambda$ is a diagonal matrix. Then we have
\begin{equation*}
\sum_{l=1}^LQ_l^2 = UW \Lambda W^TU^T.
\end{equation*}
Hence, $UW$ are the eigenvectors of $\sum_{l=1}^L Q_l^2$.

Therefore, we can estimate $U$ using the eigenvectors of the sample version $\sum_{l=1}^L Q_l^2$. Specifically, to mitigate the bias introduced by the diagonal part of $\sum_{l=1}^LA_l^2$, we conduct eigendecomposition on the following bias-adjusted sum of squares of the adjacency matrices 
$$\sum_{l=1}^L (A_l^2 - D_l),$$ where $D_l$ is an $n\times n$ diagonal matrix with $D_{l,ii}=\sum_{j=1}^n A_{l,ij}$ representing the degree of node $i$ in layer $l$. The resulting eigenvectors are denoted by $\hat{U}$.


Given $\hat{U}$, we are ready to estimate the $M_l$ matrices by recalling the decomposition of the population matrix $Q_l$ in (\ref{Q_l}). 
Specifically, we can obtain the following estimator for each layer $l$:
\begin{equation}\label{M_l}
	\hat{M_l} = \argmin\limits_{M \in \mathbb R^{K\times K}} \| A_l - \hat{U}M\hat{U}^T \| = \hat{U}^T A_l \hat{U}.
\end{equation} 

\begin{remark}
  {Similar estimation approaches for scaled connectivity matrices have been used by \cite{arroyo2021inference}, where the common eigenspace $U$ is derived based on distributed estimation techniques \citep{fan2019distributed}. By contrast, our methodology utilizes the bias-adjusted sum of squares of the adjacency matrices to estimate the eigenspace, aiming to avoid potential signal cancellation. This aggregation method was initially proposed by \cite{LeiJ2022Bias}; however, their primary focus was on the statistical consistency of community detection instead of the asymptotic normality in the estimation of connectivity matrices.}
\end{remark}



\section{Asymptotic properties}\label{asymptotic}
In this section, we present the asymptotic properties of the individual estimates of the scaled connectivity matrices $M_l (l \in [L])$ in the multi-layer SBM.

To establish the asymptotic normality of $M_l$, we need the following assumptions.
\begin{assump}\label{bal}
	The number of communities $K$ is fixed. The community sizes are balanced, that is, there exists some constant $c_1>1$ such that each community size is in $[c_1^{-1}n/K, c_1n/K]$.
\end{assump}

{Assumption \ref{bal} also indicates that there exist positive constants $c_2, c_3$, and an orthogonal matrix $Z \in \mathbb O(K)$, such that $c_2/\sqrt{n} \leq |(UZ)_{is}| \leq c_3/\sqrt{n}$ for all $i \in [n]$ and $s \in [K]$. This is referred to as eigenvector delocalization \citep{rudelson2015Delocalization, he2019local}. Indeed, by recalling the relationship $U=\Theta\Delta^{-1/2}$ and choosing the orthogonal matrix $Z$ with $c_2/\sqrt{K} \leq Z_{st} \leq c_3/\sqrt{K}$ for each $K$ and each pair $1 \leq s, t \leq K$, the delocalization of $U$ follows by Assumption \ref{bal}.}  

\begin{assump}\label{fullrank}
	The minimum eigenvalue of $\sum_{l=1}^LB_l^2$ is at least $c_4L$ for some constant $c_4>0$.
\end{assump}

Assumption \ref{fullrank} specifies the growth rate of the minimum eigenvalue of $\sum_{l=1}^LB_l^2$ in order to make the theoretical bound  of the estimated common eigenspace $\hat U$ around $U$ concrete. The linear growth rate is reasonable because when the connectivity matrices are common, Assumption \ref{fullrank} holds naturally.

\begin{assump} \label{var}
Suppose the sum of the variance of the edges satisfies $$s^2(Q_l):=\sum_{i=1}^n \sum_{j=1}^n Q_{l,ij}(1-Q_{l,ij}) = \omega(1)$$ for all $l\in [L]$.
\end{assump}

Assumption \ref{var} is critical to ensure the Lindeberg conditions for the Central Limit Theorem are met. For balanced community sizes, Assumption \ref{var} simplifies to $\sum_{s=1}^K \sum_{t=1}^K \rho   B_{l,st}(1-\rho B_{l,st}) = \omega(1/n^2)$, which suffices if $\rho= \omega(1/n^2)$. 

Before stating the next assumption, we first define the following $K(K+1)/2 \times K(K+1)/2$ matrix $\Sigma^l$,
\begin{equation}\label{sigma}
	\Sigma_{\frac{2s+t(t-1)}{2},\frac{2s'+t'(t'-1)}{2}}^l := \sum_{i=1}^{n-1}\sum_{j=i+1}^n(U_{is}U_{jt} + U_{js}U_{it} )\;(U_{is'}U_{jt'}+U_{js'}U_{it'})\;Q_{l,ij}(1 - Q_{l,ij})
\end{equation}
for each $1\leq s \leq t \leq K$ and $1\leq s'\leq t' \leq K$. Also, define $G_l = A_l - P_l$ with $P_l = \mathbb{E}(A_l) = Q_l - \text{diag}(Q_l)$ to be the noise matrix. 
As we will see, $\Sigma^l$ actually serves as the covariance matrix of the vectorized $U^TG_lU$, denoted by ${\rm vec}(U^TG_lU)\in \mathbb R^{K(K+1)/2}$, where we focus on the upper triangular part of the symmetric $U^TG_lU$. Specifically, for any $1 \leq s \leq t \leq K$, the $\frac{2s+t(t-1)}{2}$th entry in ${\rm vec}(U^TG_lU)$ corresponds to $(U^TG_lU)_{s, t}$.

\begin{assump} \label{cov eigenvalue}
	$\lambda_{\min}(\Sigma^l) = \omega(1/n^2)$ for all $l\in [L]$, where $\lambda_{\min}(\cdot)$ denotes the smallest eigenvalue. 
\end{assump}

Similar to Assumption \ref{var}, Assumption \ref{cov eigenvalue} contributes to verifying the Lindeberg conditions for the Central Limit Theorem. By Assumption \ref{var} and the discussions after Assumption \ref{bal}, we can conclude that each entry of $\Sigma_l$ is $\omega(1/n^2)$. Hence, Assumption \ref{cov eigenvalue} is stronger than Assumption \ref{var}. 

\begin{remark}
Assumptions \ref{var} and \ref{cov eigenvalue} were first used in \citet{arroyo2021inference}. 
\end{remark}

With these assumptions, we obtain the following entry-wise and vector-wise asymptotic properties of the estimated scaled connectivity matrices under the multi-layer SBM. 

\begin{theorem}\label{clt}
	Suppose Assumptions {\rm \ref{bal}} and {\rm \ref{fullrank}} hold for the multi-layer SBM, and a positive constant $c_5$ exists such that $n\rho \geq c_5\log(L+n)$. Then the estimate $\hat{M_l}$ obtained from {\rm (\ref{M_l})} has the following asymptotic properties. 
			
	{\rm (a)} If Assumption {\rm \ref{var}} holds, then for any given $l\in[L]$ and $s,t\in[K]$, we have
	\begin{equation*}
		\left(\Sigma_{\frac{2s+t(t-1)}{2},\frac{2s+t(t-1)}{2}}^l\right)^{-1/2} \left (Z\hat{M_l}Z^T - M_l - ZE_lZ^T\right)_{s, t} \rightarrow_d \mathcal N (0, 1)
	\end{equation*}
	as $n$ goes to infinity.
	
	{\rm (b)} If Assumption {\rm \ref{cov eigenvalue}} holds, then for any given $l\in[L]$, we have
	\begin{equation*}
		\left(\Sigma^l\right)^{-1/2} {\rm vec}\left (Z\hat{M_l}Z^T - M_l - ZE_lZ^T\right) \rightarrow_d \mathcal N (\bm 0, \bm I)
	\end{equation*}
	as $n$ goes to infinity. Here, $\mathcal N (\bm 0, \bm I)$ is a standard normal distribution in $K(K+1)/2$ dimensions.
 
    In (a) and (b), $Z = \arginf\limits_{Z\in \mathbb O(K)}\|U^T \hat{U} - Z\|_F$ is a $K\times K$ orthogonal matrix, and $E_l$ is a diminishing term satisfying $\|E_l\|_F \leq  c_6\left(\rho + \frac{\log^{1/2}(L+n)}{L^{1/2}} \right)$ with high probability for some positive constant $c_6$.
 
\end{theorem}


\begin{remark}
By Assumption \ref{var} and the discussions after Assumption \ref{bal}, we have $\Sigma_{\frac{2s+t(t-1)}{2},\frac{2s+t(t-1)}{2}}^l \\ = \frac{s^{2}(Q_l)}{n^2} (1 + o(1))$. Thus the smaller edge variance $s^{2}(Q_l)$ would lead to a more efficient estimate $(\hat{M}_l)_{st}$, as long as the edge variance satisfies Assumption \ref{var}. 
\end{remark}

\begin{remark}
The bias term  $\|E_l\|_F$ is negligible because it is dominated by $\|M_l\|_F\asymp n\rho$. This is also numerically verified in Section \ref{simulations}. 
\end{remark}

\begin{remark}
Compared with \citet{arroyo2021inference}, the merit of Theorem \ref{clt} lies in that we derive the non-asymptotic error bound for the bias term while their bound holds in expected value. In addition, we relax the full rank assumption of each connectivity matrix therein. 
\end{remark}

\section{Extension to multi-layer degree-corrected SBMs}\label{MDCSBM}
In this section, we extend the proposed method and the corresponding asymptotic normality results to multi-layer degree-corrected stochastic block models, which is a counterpart of the multi-layer SBM but in each layer the network is assumed to be generated from a degree-corrected SBM (DCSBM) \citep{karrer2011stochastic}.


The SBM can not capture the degree heterogeneity inherent in the networks. To address this, the DCSBM extends the standard SBM by incorporating node specific parameters, allowing degrees to vary within the same community. 

We now introduce the multi-layer DCSBM, where the layer-wise networks are generated from the DCSBM. Without specification, the notes and notation are the same as those in Section \ref{MSBM}. Define $\psi \in \mathbb{R}^n_+$  to be the degree heterogeneity parameter which measures the propensity of a node in forming edges with other nodes and is consensus among layers. For  identifiability, we assume $\max_{i\in C_k}\psi_i =1$ for all $k\in[K]$ \citep{lei2015consistency, zhang2022randomized}.
Given $\psi$, the community assignments $g_i\in[K]$, the connectivity matrix $B_l (l \in [L])$, and the network sparsity parameter $\rho$, the element $A_{l,ij}(i<j)$ is generated independently as follows
\begin{equation*}\label{adj_generate_dc}
	A_{l,ij}\sim {\rm Bernoulli}(\rho \psi_i\psi_j B_{l,g_ig_j}),
\end{equation*}
and $A_{l,ij}=A_{l,ji}$ and $A_{l,ii}=0$. The population adjacency matrix of $A_l$ is then $$Q_l=\rho {\rm diag}(\psi) \Theta B_l  \Theta^T  {\rm diag}(\psi).$$ To facilitate further analysis, we now give some additional notation. For each $k\in [K]$, define $G_k =\{1 \leq i \leq n : g_i = k\}$ as the set of nodes whose community membership is $k$. Let $\phi_k$ be an $n \times 1$ vector that matches $\psi$ on the index set $G_k$ and is zero elsewhere. Define $\Psi = {\rm diag}(\|\phi_1\|_2^2, \ldots, \|\phi_K\|_2^2)$, where each diagonal element $\|\phi_k\|_2^2$ represents the effective community size of community $k$. This allows us to reformat $Q_l$ as
\begin{equation*}
	Q_l = \rho {\rm diag}(\psi) \Theta \Psi^{-1/2} \Psi^{1/2} B_l  \Theta^T \Psi^{1/2} \Psi^{-1/2}  {\rm diag}(\psi) := UM_lU^T,
\end{equation*}
where $U = {\rm diag}(\psi) \Theta \Psi^{-1/2}$ and $M_l = \rho\Psi^{1/2} B_l \Psi^{1/2}$. It can be shown that $U^TU = I$.


As described in the case of multi-layer SBMs, we regard $U$ as the eigenspace of the following $\sum_{l=1}^LQ_l^2$,
\begin{equation*}
	\sum_{l=1}^LQ_l^2 = \sum_{l=1}^L\rho^2{\rm diag}(\psi) \Theta B_l  \Theta^T  {\rm diag}^2(\psi) \Theta B_l  \Theta^T  {\rm diag}(\psi)= U\sum_{l=1}^L\rho^2\Psi^{1/2} B_l \Psi B_l \Psi^{1/2} U^T.
\end{equation*}
Suppose $\sum_{l=1}^LB_l^2$ is of full rank, then $\sum_{l=1}^L\rho^2\Psi^{1/2} B_l \Psi B_l \Psi^{1/2}$ is of full rank as well. Similar to the case of multi-layer SBMs, we perform the eigendecomposition on the bias-adjusted sum of squares of the adjacency matrices $\sum_{l=1}^L (A_l^2 - D_l)$, where recall $D_l$ is an $n\times n$ diagonal matrix with $D_{l,ii}=\sum_{j=1}^n A_{l,ij}$ and the resulting eigenvectors are denoted by $\hat{U}$. We further obtain $\hat{M_l}$ with the help of $\hat{U}$ through \eqref{M_l}.

The following Assumptions \ref{bal_dc}, \ref{var_dc}, and \ref{cov eigenvalue dc} are needed for establishing the asymptotic normality of $\hat{M}_l$, which are counterparts of the assumptions under the multi-layer SBM.  


\renewcommand{\theassump}{E1}
\begin{assump}\label{bal_dc}
	The number of communities $K$ is fixed and $K\|\phi_k\|_2 \asymp \|\psi\|_2$ for all $k \in [K]$. 
\end{assump}
\renewcommand{\theassump}{\arabic{assump}}

\renewcommand{\theassump}{E3}
\begin{assump}\label{var_dc}
	$\min\limits_{l \in [L]}\sum_{i=1}^{n}\sum_{j=1}^nQ_{l,ij}(1 - Q_{l,ij}) = \omega(n^2/\|\psi\|_2^4)$.
\end{assump}
\renewcommand{\theassump}{\arabic{assump}}

\renewcommand{\theassump}{E4}
\begin{assump} \label{cov eigenvalue dc}
	$\lambda_{\min}(\Sigma^l) = \omega(1/\|\psi\|^4_2)$ for all $l\in [L]$, where $\lambda_{\min}(\cdot)$ denotes the smallest eigenvalue. 
\end{assump}
\renewcommand{\theassump}{\arabic{assump}}

Assumption \ref{bal_dc} requires that the node propensity parameters restricted to each community have the same order of Euclidean norm. This condition is frequently imposed in the analysis of the DCSBM; see \cite{su2020strong, jin2022optimal} and \cite{agterberg2022joint}. 
Generally, Assumptions \ref{var_dc} and \ref{cov eigenvalue dc} are more stringent than Assumptions \ref{var} and \ref{cov eigenvalue} because $\|\psi\|_2 \lesssim \sqrt{n}$; while in the special case of multi-layer SBMs where $\psi=\mathbf{1}$ and $\|\psi\|_2 \asymp \sqrt{n}$, 
Assumptions \ref{var_dc} and \ref{cov eigenvalue dc} reduce to Assumptions \ref{var} and \ref{cov eigenvalue}, respectively. With these assumptions, the entry-wise and vector-wise asymptotic normality hold for the estimated scaled connectivity matrix $\hat{M}_l$ under the multi-layer DCSBM.

\begin{theorem}\label{clt_dc}
Suppose Assumptions  {\rm \ref{bal_dc}} and {\rm \ref{fullrank}} hold for the multi-layer DCSBM, and a positive constant $c_7$ exists such that $n\rho \geq c_7\log(L+n)$. Consider the estimate $\hat{M_l}$ obtained from {\rm (\ref{M_l})}, which has the following asymptotic properties. 
		
	{\rm (a)} If Assumption {\rm \ref{var_dc}} holds, then for any given $l\in[L]$ and $s,t\in[K]$, we have
	\begin{equation*}
		\left(\Sigma_{\frac{2s+t(t-1)}{2},\frac{2s+t(t-1)}{2}}^l\right)^{-1/2} \left (Z\hat{M_l}Z^T - M_l - ZE_lZ^T\right)_{s, t} \rightarrow_d \mathcal N (0, 1)
	\end{equation*}
	as $n$ goes to infinity. 
	
	{\rm (b)} If Assumption {\rm \ref{cov eigenvalue dc}} holds, then for any given $l\in[L]$, we have
	\begin{equation*}
		\left(\Sigma^l\right)^{-1/2} {\rm vec}\left (Z\hat{M_l}Z^T - M_l - ZE_lZ^T\right) \rightarrow_d \mathcal N (\bm 0, \bm I)
	\end{equation*}
	as $n$ goes to infinity. Here, $\mathcal N (\bm 0, \bm I)$ is a standard normal distribution in $K(K+1)/2$ dimensions.
 
In (a) and (b), $Z = \arginf\limits_{Z\in \mathbb O(K)}\|U^T \hat{U} - Z\|_F$ is a $K\times K$ orthogonal matrix, and $E_l$ is a bias term satisfying $\|E_l\|_F \leq c_8\left(\frac{n^2\log^{1/2}(L+n)}{L^{1/2}\|\psi\|_2^4} + \max\{\frac{n^{3/2}\rho^{1/2}}{\|\psi\|_2^4}, \rho\} \right)$ with high probability for some positive constant $c_8$.
\end{theorem}
\begin{remark}{
$\|\psi\|_2$ reveals the degree of heterogeneity to some extent. A larger $\|\psi\|_2$ would lead to a smaller bias $\|E_l\|_F$. When $\|\psi\|_2 = \omega(n^{1/4})$, the upper bound of $\|E_l\|_F$ does not exceed $c_9\left( n^{1/2}\rho^{1/2} + \frac{n}{L^{1/2}} \right)$ up to log factors, where $c_9$ is some positive constant. In this case, when $\rho L^{1/2}=\omega(1)$, the bias term $\|E_l\|_F$ is dominated by $\|M_l\|_F\asymp n\rho$. }

\end{remark}

\begin{remark}
The discrepancy between $U$ and $\hat U$, which we establish in Lemma \ref{eigenspace bound dc}, is important for Theorem \ref{clt_dc}. It is of independent interest in evaluating the misclassification performance of the spectral clustering-based algorithm under the multi-layer DCSBM.  
\end{remark}

\section{Simulations} \label{simulations}
In this section, we conduct simulations to investigate the finite sample performance of the proposed method. In Section \ref{Bias}, we verify the vanishing behavior of the bias term $E_l$ in the formulation of asymptotic normality. In Sections \ref{IE} and \ref{GraphTest}, we test the efficacy of the derived asymptotic normality in two downstream statistical inference tasks, one for the interval estimation of the entries of the scaled connectivity matrices and the other for hypothesis testing to infer whether the multi-layer network exhibits homogeneity.

The proposed method is denoted by SCCE, namely, the abbreviation for \textbf{s}pectral \textbf{c}lustering-based method for \textbf{c}onnectivity matrix \textbf{e}stimation. 
Since studies on the asymptotic properties under multi-layer SBMs are relatively limited, we compare SCCE with the Multiple Adjacency Spectral Embedding (MASE) method studied in \cite{arroyo2021inference} and \cite{zheng2022limit}. 

\subsection{Bias evaluation} \label{Bias}
As shown in Theorems \ref{clt} and \ref{clt_dc}, $\hat M_l$ exhibits asymptotic normality with a  bias term $E_l$ for all $l \in [L]$. The theoretical results indicate that for a fixed number of nodes $n$, the bias tends to zero in Frobenius norm as the number of layers $L$ increases and the edge density $\rho$ decreases. The exact formulation for $E_l$ can be found in the proofs and \eqref{El_detail}.

In this experiment, we numerically verify the vanishing behavior of the bias $E_l$. To this end, we generate the following multi-layer SBM. Consider $n$ nodes per network with $K = 3$ communities proportional to the number of nodes via $(0.4,0.3,0.3)$. We set $B_l = \rho B^{(1)}$ for  $l\in \{1, \ldots, L/2\}$ and $B_l = \rho B^{(2)}$ for $l\in \{L/2 + 1, \ldots, L\}$, with 

{
 \begin{equation*}
		B^{(1)} = U\begin{bmatrix} 1.5 & 0 & 0 \\ 0 & 0.2 & 0 \\ 0 & 0 & 0.5 \end{bmatrix}U^T \approx \begin{bmatrix}0.675 & 0.175 & 0.46 \\ 0.175 & 0.675 & 0.46 \\ 0.46 & 0.46 & 0.85 \end{bmatrix}
	\end{equation*} 
	and
	\begin{equation*}
		B^{(2)} = U\begin{bmatrix} 1.5 & 0 & 0 \\ 0 & 0.2 & 0 \\ 0 & 0 & -0.5 \end{bmatrix}U^T \approx \begin{bmatrix} 0.175 & 0.675 & 0.46 \\ 0.675 & 0.175 & 0.46 \\ 0.46 & 0.46 & 0.85 \end{bmatrix},
	\end{equation*} 
	where 
	\begin{equation}\label{U}
		U = \begin{bmatrix}
			1/2 & 1/2 & -\sqrt{2}/2 \\ 1/2 & 1/2 & \sqrt{2}/2 \\ \sqrt{2}/2 & -\sqrt{2}/2 & 0
		\end{bmatrix}. 	
	\end{equation}
 }
 
We test how the bias $\sum_{l=1}^L \|E_l\|_F/L$ varies against the number of layers $L$. The average results over 100 replications are displayed in Figure \ref{bias_fig}. Specifically, in Figure \ref{bias_fig:a}, the number of nodes is fixed to be 500 and we test the method for various fixed $\rho\in \{0.3, 0.2, 0.1, 0.05\}$. The bias decreases as $L$ increases and $\rho$ decreases. In particular, when $\rho$ is large, our upper bound on the bias indicates that $\rho$ is the dominating term and large $L$ will not decrease the bias. The numerical result with $\rho=0.3$ coincides with this theoretical finding. In Figure \ref{bias_fig:b}, the overall edge density $\rho$ is fixed to be 0.1 and we test the results for various fixed $n\in\{200, 300, 400, 500\}$. The results are also consistent with the theoretical bound, showing that the bias decreases as the number of layers $L$ or the number of nodes $n$ increases.


\begin{figure*}[!htbp]
    \centering
    \begin{subfigure}{0.48\textwidth}
    	\includegraphics[width=\textwidth]{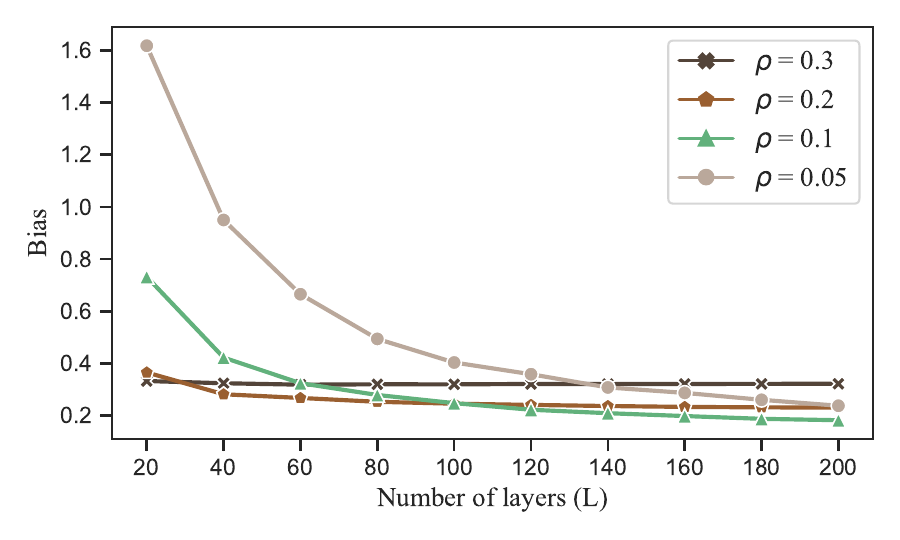}
        \vspace{-0.14\textwidth}
        \caption{}
        \label{bias_fig:a}
    \end{subfigure}
    \hspace{0.01\textwidth}
    \begin{subfigure}{0.48\textwidth}
        \includegraphics[width=\textwidth]{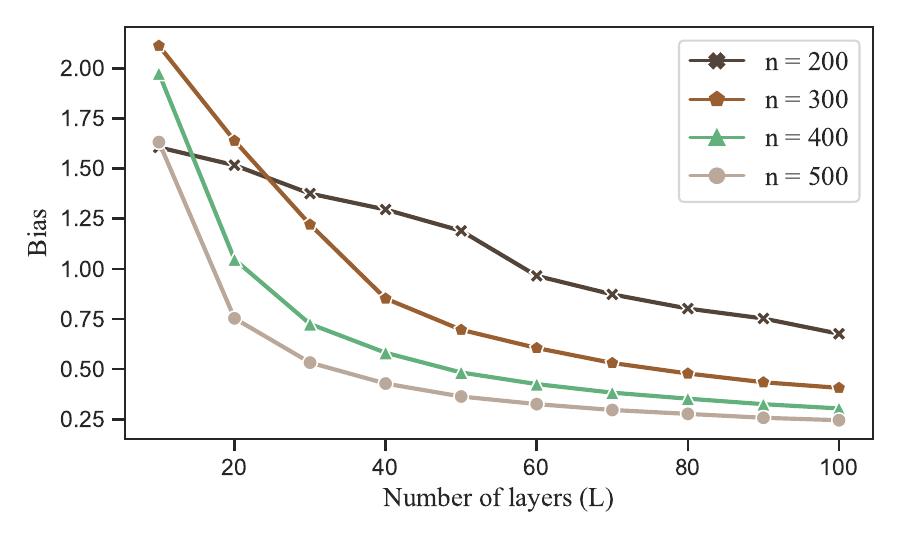}
        \vspace{-0.14\textwidth}
        \caption{}
        \label{bias_fig:b}
    \end{subfigure}
    \caption{The Frobenius norm of the bias in Theorem \ref{clt}, averaged over $L$ layers. (a) The number of nodes $n = 500$. (b) The overall edge density $\rho = 0.1$.}
    \label{bias_fig} 
\end{figure*}	

\subsection{Interval estimation} \label{IE}
The asymptotic distribution of $\hat{M}_{l,ij}$ helps to establish the interval estimate of $M_{l,st}$ for $1 \leq l \leq L$ and $1\leq s, t \leq K$. 
In particular, the interval estimate for $M_{l,st}$, as indicated by Theorems \ref{clt} and \ref{clt_dc}, is given by $$(Z\hat{M}_lZ^T)_{s, t} \pm 
u_{\alpha/2}\left(\hat{\Sigma}_{\frac{2s+t(t-1)}{2},\frac{2s+t(t-1)}{2}}^l\right)^{1/2},$$ where $u_{\alpha/2}$ is the $\alpha/2$-th percentile of the standard normal distribution, and $\hat{\Sigma}$ represents the estimated covariance matrix, which is the counterpart of \eqref{sigma} with $U$ and $Q_l$'s replaced by $\hat{U}$ and $\hat{A}_l$'s, respectively. 

In this experiment, we test the accuracy of the estimated interval using the probability of the estimated interval covering the true underlying parameter ${M}_{l,ij}$ over 200 replications. We consider two model set-ups, one for the multi-layer SBM and the other for the multi-layer DCSBM. 
In the multi-layer SBM, the network generation process is similar to that described in Section \ref{Bias}. We consider various combinations of $n\in \{100,200,300,400,500\}$ and $\rho\in \{0.3,0.2,0.1,0.05\}$. For each combination of $(n,\rho)$, we consider two different numbers of layers, $L$, set at 50 and 100, respectively. In the multi-layer DCSBM, the network generation process is also similar to that in Section \ref{Bias} with the number of nodes $n\in\{300,400,500\}$ except that we introduce the degree heterogeneity parameter $\psi$. For each $\psi_i$, we independently generate $\tilde\psi_i$ from the uniform distribution $U(2, 3)$ and set $\psi_i = \beta_n \tilde \psi_i / \|\tilde \psi\|_2$  \citep{jin2022optimal}. Here we set the corresponding $\beta_n$ to 10.4, 12 and 13.4 when the number of nodes $n$ is 300, 400 and 500, respectively. 

Tables \ref{cp_sbm} and \ref{cp_dcsbm} show the average rate of the 95\% confidence intervals covering the true parameters over 200 replications for the multi-layer SBM and multi-layer DCSBM, respectively. In particular, we calculate the average coverage rate over all the triples $(s,t,L)$. The results show that as $\rho$, $n$, and $L$ increase, the coverage rate under both multi-layer SBMs and multi-layer DCSBMs improves. The proposed method consistently outperforms the MASE over all range of the considered parameter settings. In addition, we observe that the 95\% confidence intervals do not achieve a 95\% coverage rate, which can be attributed to the presence of a bias term or to an inadequate number of nodes and network layers. Nevertheless, as the number of layers and nodes increases, the coverage rate becomes close to 95\%.


\begin{table}[!htbp] 
	\caption{The average coverage rate of the 95\% confidence intervals under multi-layer SBMs.} 
	\centering
	\label{cp_sbm}
	\begin{tabular}{cccccccccc}
		\toprule
		\multirow{2}*{} & \multicolumn{1}{c}{} & \multicolumn{2}{c}{$\rho = 0.3$} & \multicolumn{2}{c}{$\rho = 0.2$} & \multicolumn{2}{c}{$\rho = 0.1$} & \multicolumn{2}{c}{$\rho = 0.05$} \\
		\cmidrule(lr){3-4}\cmidrule(lr){5-6}\cmidrule(lr){7-8}\cmidrule(lr){9-10}
		${n}$ & $L$ & SCCE & MASE & SCCE & MASE & SCCE & MASE & SCCE & MASE \\
		\midrule
		\multirow{2}*{100} & 50 & 0.860 & 0.764 & 0.752 & 0.696 & 0.787 & 0.697 & 0.817 & 0.743\\
		& 100 & 0.906 & 0.864 & 0.846 & 0.728 & 0.811 & 0.732 & 0.854 & 0.753\\
		\midrule
		\multirow{2}*{200} & 50 & 0.920 & 0.876 & 0.902 & 0.782 & 0.705 & 0.655 & 0.726 & 0.530\\
		& 100 & 0.924 & 0.906 & 0.919 & 0.872 & 0.865 & 0.715 & 0.755 & 0.598\\
		\midrule
		\multirow{2}*{300} & 50 & 0.924 & 0.895 & 0.920 & 0.799 & 0.831 & 0.618 & 0.588 & 0.550\\
		& 100 & 0.928 & 0.917 & 0.929 & 0.885 & 0.905 & 0.756 & 0.724 & 0.636\\
		\midrule
		\multirow{2}*{400} & 50 & 0.925 & 0.917 & 0.928 & 0.832 & 0.903 & 0.456 & 0.630 & 0.467\\
		& 100 & 0.927 & 0.923 & 0.928 & 0.894 & 0.922 & 0.658 & 0.857 & 0.490\\
		\midrule
		\multirow{2}*{500} & 50 & 0.929 & 0.926 & 0.927 & 0.854 & 0.911 & 0.404 & 0.752 & 0.413\\
		& 100 & 0.929 & 0.929 & 0.927 & 0.903 & 0.925 & 0.602 & 0.886 & 0.419\\
		\bottomrule
	\end{tabular}
\end{table}

\begin{table}[!htbp] 
	\caption{The average coverage rate of the 95\% confidence intervals under multi-layer DCSBMs.} 
	\centering
	\label{cp_dcsbm}
	\begin{tabular}{cccccccccc}
		\toprule
		\multirow{2}*{} & \multicolumn{1}{c}{} & \multicolumn{2}{c}{$\rho = 0.3$} & \multicolumn{2}{c}{$\rho = 0.2$} & \multicolumn{2}{c}{$\rho = 0.1$} & \multicolumn{2}{c}{$\rho = 0.05$} \\
		\cmidrule(lr){3-4}\cmidrule(lr){5-6}\cmidrule(lr){7-8}\cmidrule(lr){9-10}
		n & L & SCCE & MASE & SCCE & MASE & SCCE & MASE & SCCE & MASE \\
		\midrule
		\multirow{2}*{300} & 50 & 0.883 & 0.555 & 0.732 & 0.537 & 0.621 & 0.438 & 0.721 & 0.581\\
		& 100 & 0.915 & 0.692 & 0.881 & 0.579 & 0.664 & 0.434 & 0.758 & 0.577\\
		\midrule
		\multirow{2}*{400} & 50 & 0.905 & 0.480 & 0.843 & 0.450 & 0.522 & 0.363 & 0.660 & 0.463\\
		& 100 & 0.924 & 0.697 & 0.905 & 0.531 & 0.700 & 0.428 & 0.679 & 0.457\\
		\midrule
		\multirow{2}*{500} & 50 & 0.916 & 0.441 & 0.880 & 0.365 & 0.480 & 0.351 & 0.561 & 0.402\\
		& 100 & 0.925 & 0.634 & 0.914 & 0.434 & 0.796 & 0.433 & 0.583 & 0.392\\
		\bottomrule
	\end{tabular}
\end{table}

\subsection{Hypothesis testing} \label{GraphTest}
In multi-layer networks, an interesting statistical inference task is to test whether there is homogeneity across different layers, specifically whether some layer-wise adjacency matrices come from the same population. In the multi-layer SBM, recall the model can be reparameterized by ($U, M_1, \ldots, M_L$), as described in (\ref{Q_l}), which implies that for any pair $(k,l)$ with $k\neq l$ and $k,l\in \{1,...,L\}$, the populations are identical, namely $Q_k=Q_l$, if and only if $M_k=M_l$. Consequently, this type of homogeneity manifests as a partition of the scaled connectivity matrices $M_l$, say $$ M_{i_1} = \cdots = M_{i_r}; \quad M_{i_{r+1}} = \cdots =  M_{i_{r+j}}; \quad \cdots. $$ Here $i_r$ is the index in $\{1, \ldots, L\}$. All pairs within a set of the partition are equal, and two $M_l$'s in different partitions are unequal. To infer the homogeneity of the multi-layer SBM, we consider the simultaneous testing of the $\binom{L}{2}$ hypotheses 
\begin{equation}\label{mt}
	H_{kl, 0}: M_k = M_{l}, \quad\quad 1 \leq k < l \leq L,
\end{equation} 
and use a Holm type step-down procedure to control the family-wise error rate \citep{lehmann2005testing}. 
The totality of acceptance and rejection statements resulting from the multiple comparison procedure may lead to a partition of the connectivity matrices. In the Holm procedure, null hypotheses are considered successively, from most significant to least significant, with further tests depending on the outcome of earlier ones. If any hypothesis is rejected at the level $\alpha' = \alpha/\binom{L}{2}$, the denominator of $\alpha'$ for the next test is $\binom{L}{2} - 1$ and the criterion continues to be modified in a stage-wise manner, with the denominator of $\alpha'$ reduced by 1 each time a hypothesis is rejected, so that tests can be conducted at successively higher significance levels. This type of multiple comparison procedure is commonly used \citep{dudoit2008multiple, noble2009does}.

The primary challenge lies in how to test the $\binom{L}{2}$ individual hypotheses, which can be facilitated using the asymptotic distribution of the estimated scaled connectivity matrices. Specifically, to test each individual hypothesis $H_{kl, 0}$ at the specified significance level, we employ the Frobenius norm of the difference between the estimated scaled connectivity matrices $T_{kl} :=\|\hat{M}_k - \hat{M}_l\|_F$ as the test statistic. The distribution of the test statistic can be determined with the help of asymptotic distribution of ${\rm vec}(\hat{M}_k - \hat{M}_l)$.

We first specify the distribution of ${\rm vec}(\hat{M}_k - \hat{M}_l)$. In Theorem \ref{clt}, we have provided the asymptotic distribution of ${\rm vec}(Z\hat{M_l}Z^T - M_l)$, with the vanishing bias term excluded. However, the dependence between $\hat{M}_k$ and $\hat{M}_{l}$ prevents a straightforward summation to obtain the asymptotic distribution of ${\rm vec}(Z\hat{M_k}Z^T - M_k - Z\hat{M_l}Z^T + M_l)$, which, under the  hypothesis $H_{kl, 0}$, reduces to the distribution of ${\rm vec}(Z(\hat{M}_k - \hat{M}_l)Z^T)$. Fortunately, by leveraging the analogous technique used in the proof of Theorem \ref{clt}, specifically the decomposition (\ref{Transform}), we can obtain the limiting distribution of ${\rm vec}(Z(\hat{M}_k- \hat{M}_l)Z^T - M_k + M_l)$. Specifically, we have
\begin{equation*}
	\left(\Sigma^k + \Sigma^l\right)^{-1/2} {\rm vec}\left(Z(\hat{M}_k- \hat{M}_l)Z^T - M_k + M_l - Z(E_k- E_l)Z^T \right) \rightarrow_d \mathcal N (\bm 0, \bm I)
\end{equation*}
for each distinct pair $k$ and $l$ satisfying $1 \leq k\neq l \leq L$, where $E_k$ and $E_l$ are negligible under certain conditions and the notation is the same as those in Theorem \ref{clt}. The conclusion holds for both the multi-layer SBM and the multi-layer DCSBM. As a result, under the hypothesis $H_{kl, 0}$ and omitting the bias terms, we can approximate the asymptotic distribution of ${\rm vec}(Z(\hat{M}_k - \hat{M}_l)Z^T)$ by $\mathcal N (\bm 0, \Sigma^k + \Sigma^l)$. With the help of asymptotic distribution of ${\rm vec}(Z(\hat{M}_k - \hat{M}_l)Z^T)$, we estimate the distribution of the test statistic $T_{kl} = \|\hat{M}_k - \hat{M}_l\|_F = \|Z(\hat{M}_k - \hat{M}_l)Z^T\|_F$ by drawing samples from $\mathcal N (\bm 0, \hat{\Sigma}^k + \hat{\Sigma}^l)$ and calculating the empirical distribution of the Frobenius norm, where $\hat{\Sigma}^k$ and $\hat{\Sigma}^l$ denote the estimated covariance matrices defined in Section \ref{IE}. The above procedure to determine the distribution of the test statistic is denoted by SCCE. 

\paragraph{\textbf{Experiment 1.}} The distribution of the test statistic $T_{kl}$ enables the testing of the individual hypothesis $H_{kl, 0}$. In this experiment, we evaluate the power of the test statistic $T_{kl}$ in testing $H_{kl, 0}$ for a given pair $k \neq l$. 
Both the multi-layer SBM and the multi-layer DCSBM are considered. We fix $L = 50$, $n=300$ and $K=3$. The number of nodes in three communities is assigned according to $(0.4,0.3,0.3)$. The overall edge density $\rho$ is fixed at 0.2. The network generation processes of both models are similar with that in Section \ref{IE} except the definitions of $B_l$'s. In this experiment, we set $B_l = \rho B^{(1)}$ for $l\in \{3, \ldots, L/2 + 1\}$, and $B_l = \rho B^{(2)}$ for $l\in \{L/2 + 2, \ldots, L\}$, where $B^{(1)}$ and $B^{(2)}$ are defined in Section \ref{Bias}. For the first layer $B_1 = \rho B^{(1)}$, and for the second layer $B_2 = \rho B^{(2')}$, the $B^{(2')}$ is defined to be the same as $B^{(1)}$ except for the first entry $B^{(2')}_{11}$. We vary $B^{(2')}_{11}$ to obtain a sequence of $B_2$ matrices. Under these parameter settings, we test the individual hypothesis $H_{12, 0}: M_1 = M_{2}$ at a specified significance level. 

\begin{figure*}[!htbp]
    \centering
    \begin{subfigure}{0.48\textwidth}
    	\includegraphics[width=\textwidth]{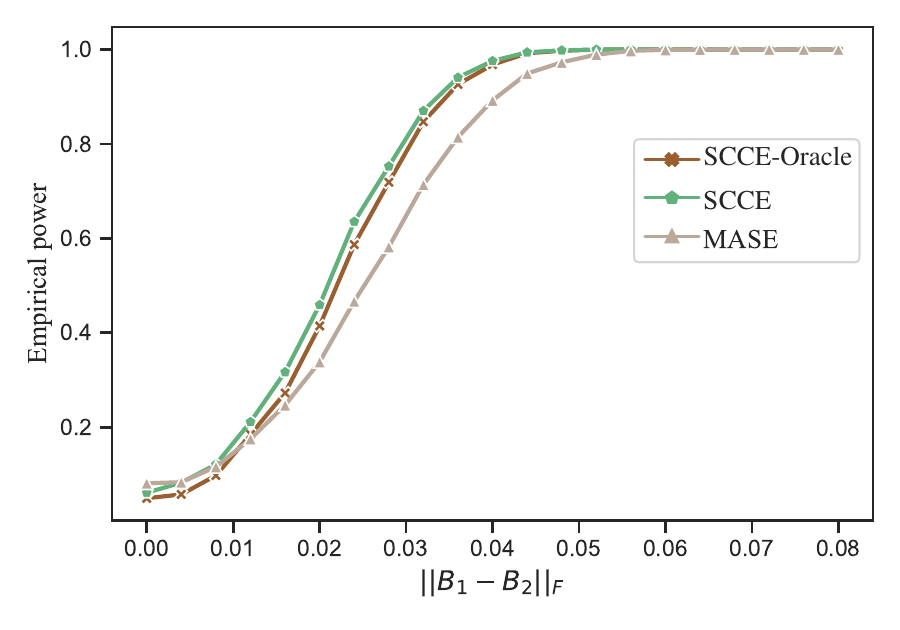}
        \vspace{-0.14\textwidth}
        \caption{multi-layer SBM}
    \end{subfigure}
    \hspace{0.01\textwidth}
    \begin{subfigure}{0.48\textwidth}
        \includegraphics[width=\textwidth]{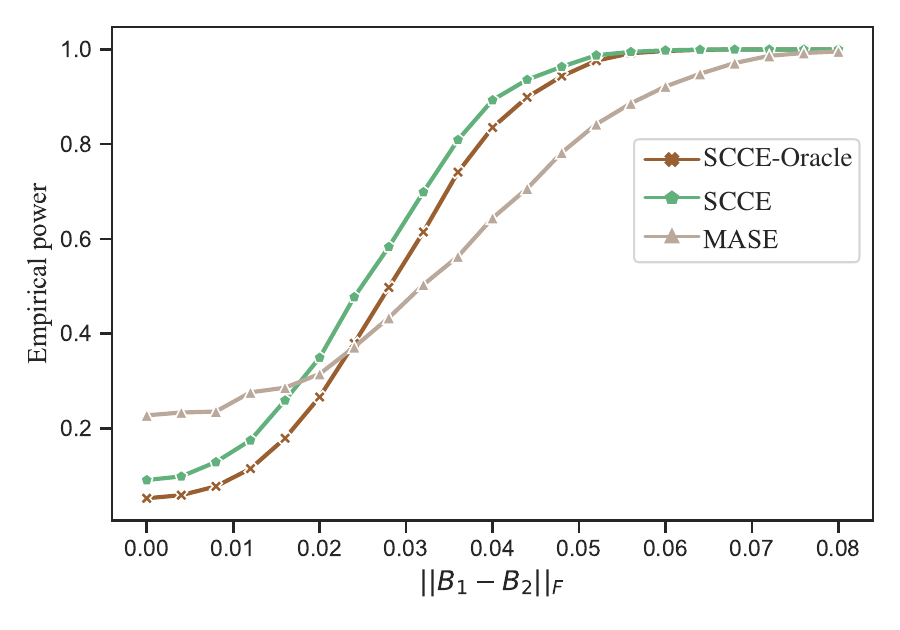}
        \vspace{-0.14\textwidth}
        \caption{multi-layer DCSBM}
    \end{subfigure}
    \caption{Empirical power to reject the hypothesis that the two different network layers have an identical population matrix at a 0.05 significance level under (a) multi-layer SBMs and (b) multi-layer DCSBMs. Empirical power is measured as the proportion of trials for which  the test identified a significant result.}
    \label{test_fig}
\end{figure*}

Figure \ref{test_fig} shows the averaged empirical power to reject the individual hypothesis $H_{12, 0}$ at a 0.05 significance level with $\|B_1-B_2\|_F$ increasing over 100 replications. 
In our comparisons, we compare SCCE with a counterpart SCCE-Oracle, which differs from SCCE primarily in that it does not use the asymptotic distribution of ${\rm vec}(Z(\hat{M}_k - \hat{M}_l)Z^T)$. Instead, SCCE-Oracle generates various samples from the population to obtain various $\hat{M}_k-\hat{M}_l$ and uses the empirical distribution to approximate the real distribution of the test statistic $T_{kl}$.  Another comparison method, MASE, is defined similarly with SCCE except that we use the estimator $\hat{M}_l$ and the asymptotic distribution derived in \citet{arroyo2021inference}. 
The results in Figure \ref{test_fig} show that the proposed method SCCE has power close to that of SCCE-Oracle, demonstrating the accuracy of the asymptotic distributions. In addition, the proposed method SCCE has higher power than MASE under both network models, showing the efficacy of the proposed method in hypothesis testing. 

\paragraph{\textbf{Experiment 2.}} We shall now consider the simultaneous testing of the $\binom{L}{2}$ hypotheses \eqref{mt} by means of a Holm type step-down procedure for inferring homogeneity of multi-layer networks. For this purpose, we consider the multi-layer SBM with two distinct connectivity matrices, where all connectivity matrices constitute a partition as follows: $$ M_1 = M_2 = \cdots = M_{L/2}; \quad\quad M_{L/2+1} =  M_{L/2+2} = \cdots = M_L.$$ All pairs within a set of the partition are equal, and two $M_l$'s in different partitions are unequal. It is clear that the overall network is heterogeneous, while the network layers within each partition exhibit homogeneity.
Specifically, we set $B_l = \rho B^{(1)}$ for $l\in \{1, \ldots, L/2\}$ and $B_l = \rho B^{(2)}$ for $l\in \{L/2 + 1, \ldots, L\}$, with 
 	\begin{equation*}
		B^{(1)} = U\begin{bmatrix} 1 & 0 & 0 \\ 0 & 0.4 & 0 \\ 0 & 0 & 0.1 \end{bmatrix}U^T \approx \begin{bmatrix}0.4 & 0.3 & 0.212 \\ 0.3 & 0.4 & 0.212 \\ 0.212 & 0.212 & 0.7 \end{bmatrix}
	\end{equation*} 
	and
	\begin{equation*}
		B^{(2)} = U\begin{bmatrix} 1 & 0 & 0 \\ 0 & 0.4 & 0 \\ 0 & 0 & -0.1 \end{bmatrix}U^T \approx \begin{bmatrix} 0.3 & 0.4 & 0.212 \\ 0.4 & 0.3 & 0.212 \\ 0.212 & 0.212 & 0.7  \end{bmatrix},
	\end{equation*} 
	where $U$ is defined as in \eqref{U}. We set $L = 20$, $n=500$ and $K=3$, with the number of nodes in three communities proportioned at $(0.4,0.3,0.3)$. The overall edge density $\rho$ is fixed at 0.2. Recall the definition of $M_l$, the difference in $M_l$ between different partitions is now given by $\rho\|B^{(1)} - B^{(2)}\|_F = 0.04$. 
	
Figure \ref{mtp_simulation} presents the outcomes of the multiple comparisons conducted using a Holm type step-down procedure, where olive green signifies acceptance and white signifies rejection of the individual hypotheses. We control the Holm procedure to ensure that the family-wise error rate is no bigger than $\alpha = 0.05$. In the results of the proposed method SCCE, two distinct blocks are observed within which all individual hypotheses are simultaneously accepted, demonstrating the presence of homogeneity within these blocks. This is consistent with our experimental setup. However, MASE cannot accurately infer the true homogeneity. Moreover, in cases where the true hypothesis is false, the proposed method SCCE almost always rejects them.

\begin{figure*}[!htbp]
    \centering
    \begin{subfigure}{0.4\textwidth}
        \includegraphics[width=\textwidth]{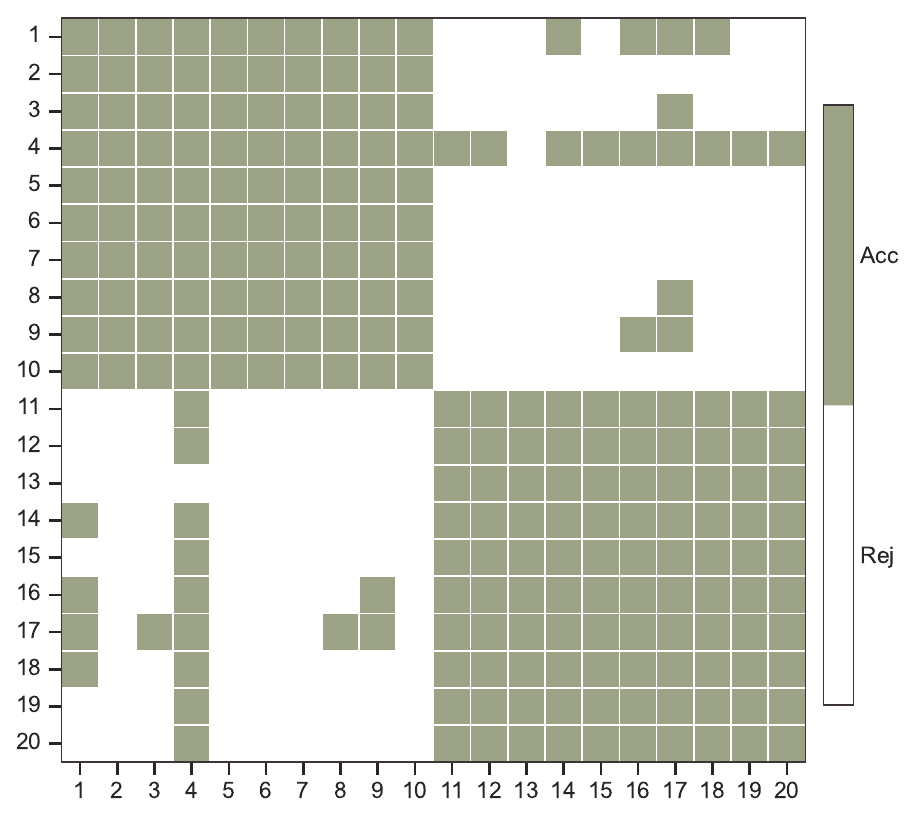}
        \vspace{-0.12\textwidth}
        \caption{SCCE}
    \end{subfigure}
    \hspace{0.06 \textwidth}
    \begin{subfigure}{0.4\textwidth}
    	\includegraphics[width=\textwidth]{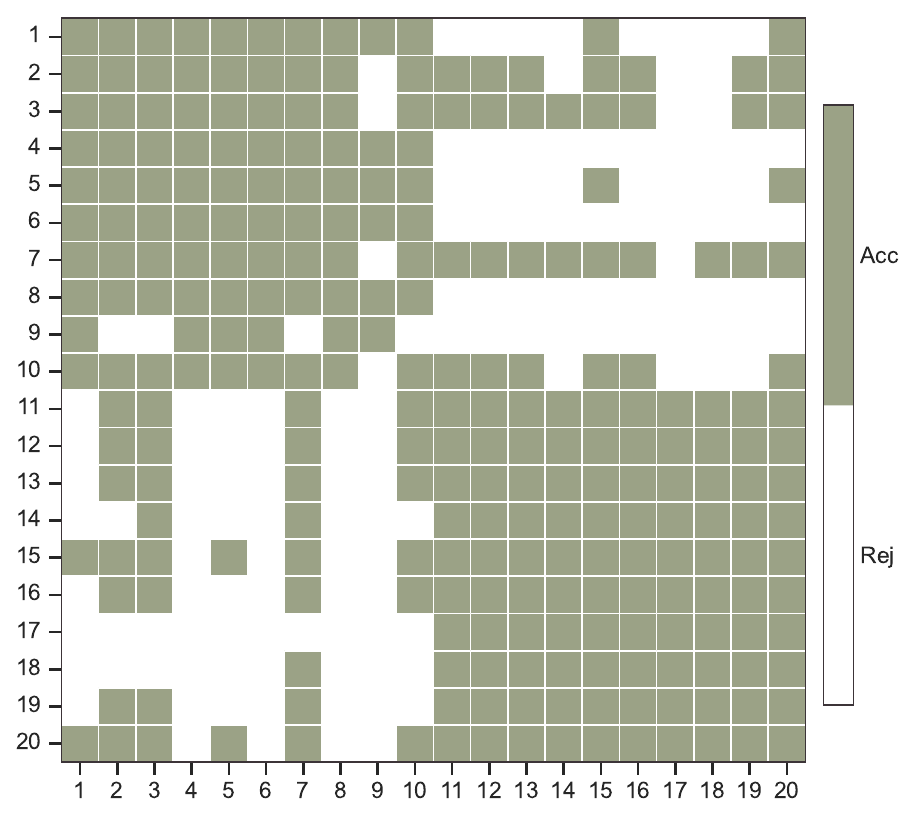}
        \vspace{-0.12\textwidth}
        \caption{MASE}
    \end{subfigure}
    \caption{Results of the multiple comparisons conducted using a Holm type step-down procedure, where olive green signifies the acceptance of the individual hypotheses.}
    \label{mtp_simulation}
\end{figure*}

\section{Real data analysis}\label{Real data analysis}
In this section, we use the proposed method SCCE to measure the homogeneity across all layers of a real multi-layer network, specifically the WFAT dataset \citep{de2015structural}. For this purpose, we simultaneously test the $\binom{L}{2}$ hypotheses $$H_{kl, 0}: M_k = M_{l}, \quad\quad 1 \leq k < l \leq L,$$ using a Holm type step-down procedure. Next, we provide the data description and the results for the real dataset.

The WFAT dataset, sourced from the Food and Agriculture Organization (FAO) of the United Nations, contains annual trade records for over 400 food and agricultural products across all countries globally. Our analysis focuses specifically on 2020 trade data for cereals, a category defined by the FAO and detailed in Table \ref{items}. We constructed a multi-layer network using this data. Each layer represents a different type of cereal, with nodes representing countries and edges within each layer representing the trade relationships with respect a specific cereal. In particular, a trade link between countries is established if the trade value for a cereal exceeds $\$2000$. To ensure network connectivity, we excluded countries with a total degree over 37 layers less than 23. This guarantees that each node is linked to at least one other node in at least half of the layers. The resulting network consists of 37 layers, each containing 114 common nodes, denoted by $A_l, l \in \{1,\ldots,37\}$. The five continents, including America, Africa, Asia, Europe, and Oceania, suggest the choice $K=5$ \citep{noroozi2022sparse}.

Figure \ref{testing for real data} shows the results of the multiple comparisons conducted using a Holm type step-down procedure, where olive green indicates the acceptance and white denotes the rejection of the individual hypotheses. We control the Holm procedure to ensure that the family-wise error rate is no greater than $\alpha = 0.05$. Additionally, the $p$-value matrix for individual hypothesis tests across all network layers is shown. In Figure \ref{mtp_holm}, several distinct blocks are observed where all individual hypotheses within each block are simultaneously accepted, demonstrating the presence of homogeneity within the multi-layer network.
For instance, layers 5 to 14, corresponding to `Germ of maize', `Triticale', `Cereals n.e.c.', `Millet', `Sorghum', `Bran of maize', `Bran of cereals n.e.c.', `Flour of mixed grain', `Flour of rice', and `Rye', exhibit no rejections of the individual hypotheses in any pairwise comparisons. This outcome suggests a significant similarity in trade patterns among these cereal products. Notably, most of these cereals are consumed primarily by Asia, which might lead to simultaneous acceptance of the individual hypotheses. This observation suggests the need for an integrated analysis of the trade patterns of these cereals.

\begin{figure*}[!htbp]
    \centering
    \begin{subfigure}{0.48\textwidth}
        \includegraphics[width=\textwidth]{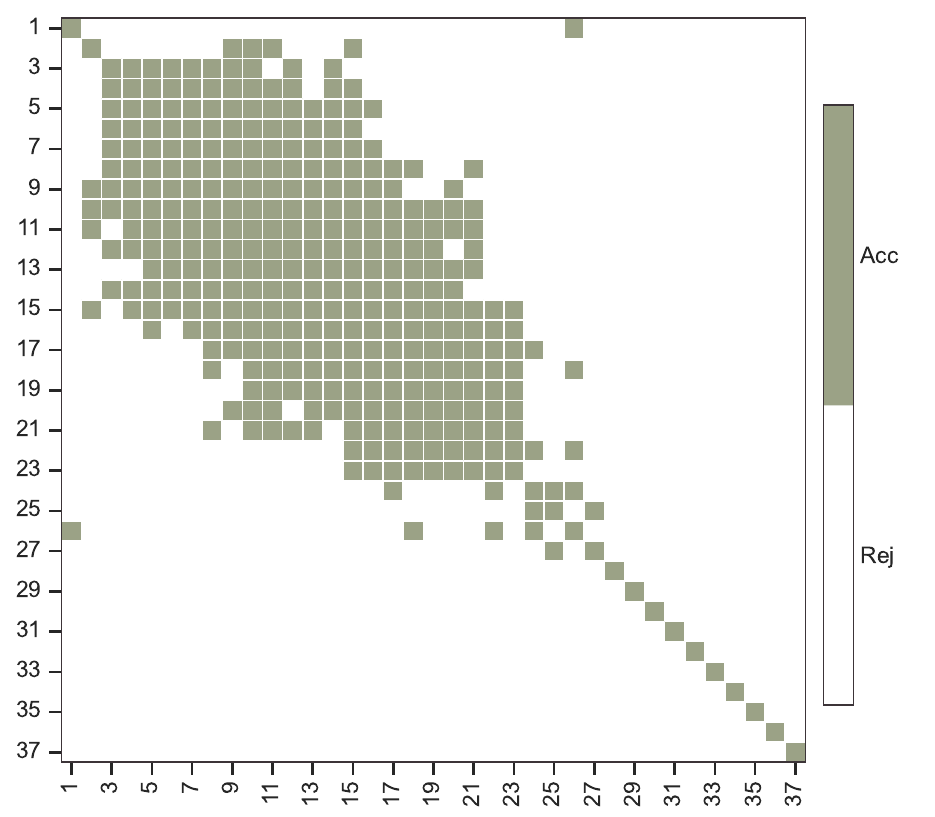}
        \vspace{-0.12\textwidth}
        \caption{multiple comparisons}
        \label{mtp_holm}
    \end{subfigure}
    \hspace{0.01 \textwidth}
    \begin{subfigure}{0.48\textwidth}
    	\includegraphics[width=\textwidth]{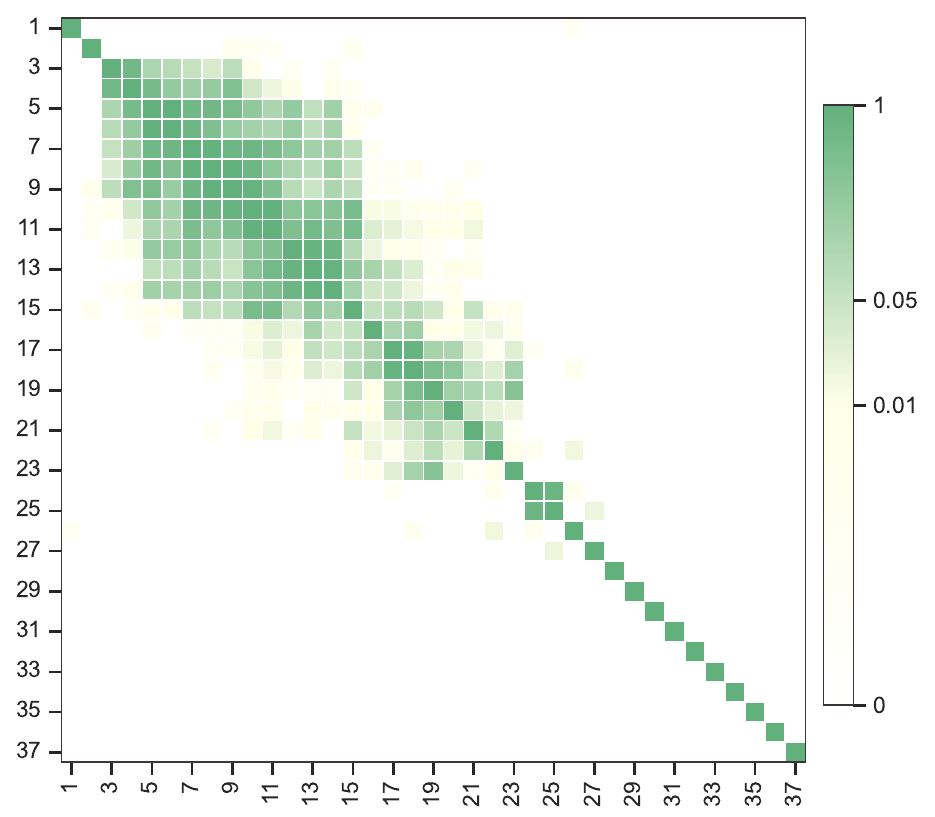}
        \vspace{-0.12\textwidth}
        \caption{$p$-values matrix}
        \label{testing for real data trade}
    \end{subfigure}
    \caption{(a) Results of the multiple comparisons conducted using a Holm type step-down procedure for the WFAT data,  with olive green indicating acceptance of the individual hypotheses. (b) Matrix of $p$-values for the individual hypothesis tests corresponding to all pairs of network layers, where the layer $l$ correspond to the cereal numbered $l$ in Table \ref{items}.}
    \label{testing for real data}
\end{figure*}	


\section{Conclusion}\label{conclusion}
In this paper, we focused on the asymptotic normality in the estimation of the scaled connectivity matrix under multi-layer SBMs and multi-layer DCSBMs. To this end, we first leveraged the leading eigenvectors of the bias-adjusted sum of squares of the adjacency matrices to estimate the common eigenspaces across layers \citep{LeiJ2022Bias}. Building on this, we further derived efficient estimates for the layer-wise scaled connectivity matrices. These estimates are shown to satisfy asymptotic normality under mild additional assumptions. To corroborate the theoretical results, we conducted a series of simulations to apply the derived asymptotic normality to statistical inference tasks, including interval estimation and hypothesis testing. Finally, we applied the method to a real-world network dataset to infer the population homogeneity within the multi-layer network.

There are many ways to extend the content of this paper. First, we assume the connectivity matrices are different among layers. However, in practice, the network layers are often grouped such that connectivity matrices are the same within each group \citep{jing2021community,fan2022asymptotic}. It is thus meaningful to make use of all the network layers within each group to estimate the connectivity matrices, which would lead to improved estimation accuracy and a weakened requirement for network sparsity. Second, we focused on the undirected multi-layer networks. It is of great importance to extend the method and theory to multi-layer directed networks \citep{zheng2022limit, su2023}. Third, we assumed the number of communities is known and fixed. In the future, it is important to study the estimation of the number of communities under multi-layer SBMs. 

\appendix
\def\theequation{A\arabic{section}.\arabic{equation}}
\def\thesection{A\arabic{section}}
\def\thesubsection{A.\arabic{subsection}} 
\def\thetheorem{A\arabic{theorem}}
\def\thelemma{A\arabic{lemma}}
\def\thefigure{A\arabic{figure}}
\def\thetable{A\arabic{table}}

\section*{Appendix} \label{appendix}
\renewcommand{\theequation}{A\arabic{equation}}
\setcounter{equation}{0}

Appendix \ref{Technical lemmas} provides the technical lemmas that are needed to prove the central limit theorems. Appendix \ref{MainProofs} includes the proof of main theorems in the main text. Appendix \ref{AuxiliaryLemmas} contains auxiliary lemmas. Appendix \ref{additional real data} provides the additional results for real data analysis.

\subsection{Technical lemmas} \label{Technical lemmas}
\begin{lemma} \label{eigenspace bound}
Suppose the multi-layer SBM satisfies Assumptions \rm{\ref{bal}} and \rm{\ref{fullrank}}. Let $\hat{U}$ denotes the leading eigenvectors of $\sum_{l=1}^L (A_l^2 - D_l)$. If $n\rho \geq c_1\log(L+n)$ for some positive constant $c_1$, then 
	 \begin{eqnarray*}
		\|\hat{U}-UZ\|_F \leq c_2 \left(\frac{1}{n} + \frac{\log^{1/2}(L+n)}{L^{1/2}\rho^{1/2}n^{1/2}} \right)
	\end{eqnarray*}
	with probability at least $1-O((L+n)^{-1})$ for some constant $c_2 > 0$. 
\end{lemma}

\begin{proof} The bias-adjusted sum-of-squares of the adjacency matrices $\sum_{l=1}^L (A_l^2 - D_l)$ can be decomposed into a signal term and for noise terms
\begin{equation*}
	\sum_{l=1}^L\left(A_l^2 - D_l\right) = \sum_{l=1}^LQ_l^2 + N_1 + N_2 + N_3 + N_4, 
\end{equation*}
where $N_1 = \mbox{diag}^2(Q_l) - Q_l\mbox{diag}(Q_l) - \mbox{diag}(Q_l)Q_l$, $N_2 = \sum_{l=1}^L\left(P_lG_l^T + G_lP_l^T\right)$, $N_3 = \sum_{l=1}^LG_l^2 - \mbox{diag}(\sum_{l=1}^LG_l^2)$, and $N_4 = \mbox{diag}(\sum_{l=1}^LG_l^2) - \sum_{l=1}^LD_l$. Recall that in the context of multi-layer SBMs, $Q_l=\rho \Theta B_l  \Theta^T$, $P_l = Q_l - \mbox{diag}(Q_l)$ and $G_l = A_l - P_l$. 

According to Theorem 1 of \cite{LeiJ2022Bias}, the minimum eigenvalue of the signal term $\sum_{l=1}^LQ_l^2$ is lower bounded by $cL\rho^2n^2$ for some positive constant $c$. Note that we use $c$ to represent the generic constant and it may be different from line to line. We proceed to establish upper bounds for the spectral norms of all noise terms. The first noise term satisfies $\|N_1\|_2 \leq cLn\rho^2.$ Regarding $N_2$, under the condition $n\rho \geq c_1 \log(L+n)$, we have $\|N_2\|_2 \leq cL^{1/2}n^{3/2}\rho^{3/2}\log^{1/2}(n+L)$ with high probability. For $N_3$, by Theorem 4 of \cite{LeiJ2022Bias}, under the condition $n\rho \geq c_1 \log(L+n)$, we have $\|N_3\|_2 \leq cL^{1/2}n\rho \log(L+n)$ with  probability at least $1-O((L+n)^{-1})$. Finally, for $N_4$, given $(\sum_{l=1}^LG_l^2)_{ii} \leq Ln\rho^2 + D_{l,ii}$ for all $i \in [m]$, it follows that $\|N_4\|_2 \leq Ln\rho^2$.

Integrating the bounds of all terms, we have with high probability,
\begin{equation*}
	\frac{\|N_1 + N_2 + N_3 + N_4\|_2}{\lambda_K (\sum_{l=1}^LQ_l^2)} \leq c \frac{Ln\rho^2 + L^{1/2}n^{3/2}\rho^{3/2}\log^{1/2}(L+n)}{L\rho^2n^2}. 
\end{equation*}
Let $\widehat{U}$ be the $n\times K$ matrices that contain the leading eigenvectors of $\sum_{l=1}^L\left(A_l^2 - D_l\right)$. In accordance with Proposition 2.2 of \cite{vu2013minimax} and the Davis-Kahan sin$\Theta$ theorem (Theorem VII.3.1 of \cite{Bhatia1997matrix}), there exists a $K \times K$ orthogonal matrix $Z$ such that
\begin{equation*}
	\|\widehat{U} - UZ\|_F  \leq  \sqrt{K}\|\widehat{U} - UZ\|_2  \leq c \left(\frac{1}{n} + \frac{\log^{1/2}(L+n)}{L^{1/2}\rho^{1/2}n^{1/2}}\right).
\end{equation*} 
The proof is completed.
\end{proof}

\begin{lemma}\label{eigenspace bound dc}
Suppose the multi-layer DCSBM satisfies Assumptions  {\rm \ref{bal_dc}} and {\rm \ref{fullrank}}. Let $\hat{U}$ denotes the leading eigenvectors of $\sum_{l=1}^L (A_l^2 - D_l)$. If $n\rho \geq c_1\log(L+n)$ for some positive constant $c_1$, then
	\begin{equation*}
		\|\hat U - UZ\|_F \leq c_2 \left(\frac{n}{\|\psi\|_2^4} + \frac{n^{3/2}\log^{1/2}(L+n)}{L^{1/2}\rho^{1/2}\|\psi\|_2^4}\right)
	\end{equation*}
with probability at least $1-O((L+n)^{-1})$ for some positive constant $c_2$.
\end{lemma}

\begin{proof} Following the approach used in the proof of Lemma \ref{eigenspace bound}, we separately bound the signal term and the noise terms. Recall that in the context of multi-layer degree-corrected SBMs, $P_l = Q_l - \mbox{diag}(Q_l)$, $Q_l=\rho \mbox{diag}(\psi) \Theta B_l  \Theta^T  \mbox{diag}(\psi)$ and $G_l = A_l - P_l$. Given the balanced effective community sizes assumption and the fixed number of communities, the minimum eigenvalue of $\Psi$ is lower bounded by $c \|\psi\|_2^2 $ for some positive constant $c$. Then we have
\begin{eqnarray*}
	\sum_{l=1}^LQ_l^2 & = & \rho^2 \tilde\Theta \Psi^{1/2} \left[\sum_{l=1}^L B_l \Psi B_l\right] \Psi^{1/2} \tilde\Theta^T \\
	& \succeq & c\rho^2\|\psi\|_2^2  \tilde\Theta \Psi^{1/2} \sum_{l=1}^L B_l^2 \Psi^{1/2} \tilde\Theta^T \\
	& \succeq & cL\rho^2\|\psi\|_2^4 \tilde\Theta \tilde\Theta^T,
\end{eqnarray*}
where $\succeq$ denotes the Loewner partial order, in particular, let $A$ and $B$ be two Hermitian matrices of order $p$, we say that $A \succeq B$ if $A - B$ is positive semi-definite. The latter $\succeq$ is derived from Assumption \ref{fullrank}. Recall that $\tilde\Theta$ is an orthogonal column matrix, the smallest eigenvalue of $\sum_{l=1}^LQ_l^2$ is lower bounded by $cL\rho^2\|\psi\|_2^4$. For the bounds of the spectral norms of all noise terms, they are consistent with those established in Lemma \ref{eigenspace bound}.

Combining the bounds of all terms, Proposition 2.2 of \cite{vu2013minimax} and the Davis-Kahan sin$\Theta$ theorem, there exists a $K \times K$ orthogonal matrix $Z$ such that
\begin{equation*}
	\|\widehat{U} - UZ\|_F  \leq  \sqrt{K}\|\widehat{U} - UZ\|_2  \leq c \frac{Ln\rho^2 + L^{1/2}n^{3/2}\rho^{3/2}\log^{1/2}(L+n)}{L\rho^2\|\psi\|_2^4}
\end{equation*} 
holds with high probability. The proof is completed.
\end{proof}


\subsection{Main proofs}\label{MainProofs}
\subsubsection*{Proof of Theorem \ref{clt}}
Recall that $Q_l = \rho\Theta B_l\Theta^T$ and the noise matrix $G_l = A_l - P_l$, where $P_l = Q_l - \text{diag}(Q_l)$ for all $l \in [L]$. Then the estimator $\hat{M_l}$ can be rearranged as
\begin{eqnarray}\label{El_detail}
	\hat{M_l} &=& \hat{U}^TQ_l\hat{U} + \hat{U}^TG_l\hat{U} - \hat{U}^T{\rm diag} (Q_l)\hat{U} \nonumber\\
	&=& \hat{U}^TUM_lU^T\hat{U} + \hat{U}^TG_l\hat{U} - \hat{U}^T{\rm diag} (Q_l)\hat{U} \nonumber\\
	&=&	(\hat{U}^TU - Z^T)M_l(U^T\hat{U}) + Z^TM_l(U^T\hat{U}-Z) + Z^TM_lZ\nonumber\\
	& & + ~(\hat{U}^T - Z^TU^T)G_l(\hat{U}-UZ) + Z^TU^TG_l(\hat{U}-UZ) - (Z^TU^T - \hat{U}^T)G_lUZ \nonumber\\
	& &	+ ~ Z^TU^TG_lUZ - \hat{U}^T{\rm diag} (Q_l)\hat{U}\nonumber\\
	&=&	Z^TM_lZ + Z^TU^TG_lUZ + E_l,
\end{eqnarray}
where $E_l$ in the final equality represents all terms from the penultimate equality excluding $Z^TM_lZ$ and $Z^TU^TG_lUZ$. Here, $Z = \arginf\limits_{Z\in \mathbb O(K)}\|U^T \hat{U} - Z\|_F$ is a $K\times K$ orthogonal matrix. Thus,
\begin{equation}\label{Transform}
	Z\hat{M_l}Z^T - M_l - ZE_lZ^T = U^TG_lU.
\end{equation}

We initiate our argument by establishing that $ZE_lZ^T$ tends towards zero as the number of nodes and layers increase, and the edge density decreases. We accomplish this by exerting control over each term in $E_l$.
\begin{eqnarray*}
	\|E_l\|_F &\leq& \|(\hat{U}^TU - Z^T)M_l(U^T\hat{U})\|_F + \|Z^TM_l(U^T\hat{U}-Z)\|_F\\
	& & +~\|(\hat{U}^T - Z^TU^T)G_l(\hat{U}-UZ)\|_F + 2\|Z^TU^TG_l(\hat{U}-UZ)\|_F \\
	& & +~\|\hat{U}^T{\rm diag} (Q_l)\hat{U}\|_F \\
	& \leq & \|\hat{U}^TU - Z^T\|_F\|M_l\|_2 + \|(U^T\hat{U}-Z)\|_F\|M_l\|_2 \\
	& & +~ \|\hat{U}-UZ\|_F^2\|G_l\|_2 + 2\|\hat{U}-UZ\|_F\|G_l\|_2 +\|{\rm diag}(Q_l)\|_2\|\hat U\|_F\\
	& = & 4(\|M_l\|_2 + \|G_l\|_2)\|\hat{U}-UZ\|_F^2 + 2\|\hat{U}-UZ\|_F\|G_l\|_2  + \sqrt{K}\|{\rm diag}(Q_l)\|_2.
\end{eqnarray*}
The upper bound of $\|\hat{U}-UZ\|_F$ is provided in Lemma \ref{eigenspace bound}. With regards to the term $M_l$, it holds that $\|M_l\|_2 = \|UM_lU^T\|_2 = \|Q_l\|_2$. For $G_l$, applying Lemma \ref{Spectral bound}, we have 
\begin{equation*}
	\|G_l\|_2 \leq c\sqrt{\|Q_l\|_{1, \infty}}
\end{equation*}
with probability at least $1 - O(n^{-1})$. Note that we use $c$ and $c'$ to represent the generic positive constants and it may be different from line to line. Consequently, it can be deduced that 
\begin{eqnarray*}
	\|ZE_lZ^T\|_F &\leq& c (\|Q_l\|_2 + \sqrt{\|Q_l\|_{1, \infty}})\left(\frac{1}{n^2} + \frac{\log(L+n)}{Ln\rho} \right) \nonumber\\
	& & ~+ c  \sqrt{\|Q_l\|_{1, \infty}} \left(\frac{1}{n} + \frac{\log^{1/2}(L+n)}{L^{1/2}n^{1/2}\rho^{1/2}} \right) + \sqrt{K}\rho
\end{eqnarray*}
with high probability. Additionally, given that every element of $Q_l$ is at most $\rho$, both $\|Q_l\|_2$ and $\|Q_l\|_{1,\infty}$ do not exceed $n\rho$. Thus, we can deduce that $\|ZE_lZ^T\|_F \leq c\left(\rho + \frac{\log^{1/2}(L+n)}{L^{1/2}} \right)$. As the number of layers $L$ increase, and the overall edge probability decrease, $ZE_lZ^T$ tends to vanish. 

Consequently, demonstrating the central limit theorem can be reduced to verifying the asymptotic normality of $U^TG_lU$.
The covariance of any pair $(U^TG_lU)_{s, t}$ and $(U^TG_lU)_{s', t'}$, where $1\leq s\leq t \leq K$ and $1\leq s'\leq t' \leq K$, is
\begin{eqnarray*}
	&   &{\rm Cov}\left((U^TG_lU)_{s, t},\;(U^TG_lU)_{s', t'}\right) \nonumber\\
	& = &{\rm Cov}\left(\sum_{1\leq i, j \leq n}U_{is}G_{l,ij}U_{jt},\;\sum_{1\leq i, j \leq n}U_{is'}G_{l,ij}U_{jt'}\right) \nonumber\\
	& = &\sum_{i=1}^{n-1}\sum_{j=i+1}^n(U_{is}U_{jt} + U_{js}U_{it} )\;(U_{is'}U_{jt'}+U_{js'}U_{it'})\;{\rm Cov}\left(G_{l,ij}, \;G_{l,ij}\right) \nonumber\\
	& = &\sum_{i=1}^{n-1}\sum_{j=i+1}^n(U_{is}U_{jt} + U_{js}U_{it} )\;(U_{is'}U_{jt'}+U_{js'}U_{it'})\;Q_{l,ij}(1 - Q_{l,ij}).
\end{eqnarray*}
That is the $\Sigma_{\frac{2s+t(t-1)}{2},\frac{2s'+t'(t'-1)}{2}}^l$ term defined in (\ref{sigma}). Regarding the expectation term, given that $\mathbb E(G_l) = 0$, it is straightforward to see that each element of $U^TG_lU$ has a zero-mean, as each element sums terms with zero-mean. In other words, we have $\mathbb E\left((U^TG_lU)_{s, t} \right) = 0$ for all $1\leq s \leq t \leq K$.
	
\textbf{Part (a):}  Note that
\begin{equation}\label{UGlU_st}
	(U^TG_lU)_{s, t} = \sum_{i=1}^{n-1}\sum_{j=i+1}^n(U_{is}U_{jt} + U_{js}U_{it})\; G_{l,ij} := \sum_{i=1}^{n-1}\sum_{j=i+1}^n F_{l,ij},
\end{equation}
where we let $F_{l,ij}$ denote  $(U_{is}U_{jt} + U_{js}U_{it})\;G_{l,ij}$. Given $l$, it is clear that the $F_{l,ij}$ terms are zero-mean and independent of each other for all $1\leq i<j \leq n$. By the definition of $U$ and according to Assumption \ref{bal}, we have $|F_{l,ij}| \leq c'/n$ for all $l\in [L]$ and $1\leq i<j \leq n$, where $c'$ is some positive constant. 
To affirm the Lindeberg condition, we proceed by calculating the sum of ${\rm Var}(F_{l,ij})$ for all $1\leq i<j \leq n$. Given Assumption \ref{var} and the premise of balanced community sizes, it follows that ${\rm Var}\left((U^TG_lU)_{s, t}\right) = \omega(1/n^2)$. Therefore, for any $\varepsilon > 0$ and sufficiently large $n$, we can state that $|F_{l,ij}| < \varepsilon \left[{\rm Var}\left((U^TG_lU)_{s, t}\right)\right]^{1/2}$.  This leads to the Lindeberg condition as given by
\begin{equation}\label{lindeberg}
	\lim\limits_{n\rightarrow \infty}\frac{1}{{\rm Var}\left((U^TG_lU)_{s, t}\right)}\sum_{i=1}^{n-1}\sum_{j=i+1}^n \mathbb E \left(|F_{l,ij}|^2 \mathbb I_{\{ |F_{l,ij}| > \varepsilon [{\rm Var}((U^TG_lU)_{s, t})]^{1/2}\}} \right) = 0
\end{equation}
is satisfied. It should be emphasised that here $U$, $G_l$ and $F_{l,ij}$ are all related to $n$. For simplicity, we have not notated the $n$ display. By the Lindeberg-Feller Central Limit Theorem, we have 
\begin{equation*}
	\left(\Sigma_{\frac{2s+t(t-1)}{2},\frac{2s+t(t-1)}{2}}^l\right)^{-1/2} (U^TG_lU)_{s, t} \rightarrow_d \mathcal N (0, 1)
\end{equation*}
for each $l\in[L]$. The claim follows by combining the above discussion with (\ref{Transform}).

\textbf{Part (b):} To complete the central limit theorem for vectorlized $U^TG_lU$, we first show that this term can be expressed as a summation of independent zero-mean random variables. Let $\widetilde{\Sigma}^l:= (\Sigma^l) ^ {-1/2}$ and consider an arbitrary vector $x\in \mathbb R^{K(K+1)/2}$.
\begin{eqnarray}\label{ver_UGlU}
	& &x^T(\Sigma^l)^{-1/2}{\rm vec}(U^TG_lU)\nonumber\\
	&=& \sum_{1\leq s\leq t \leq K}x_{\frac{2s+t(t-1)}{2}} \sum_{1\leq s'\leq t' \leq K} \tilde{\Sigma}_{\frac{2s+t(t-1)}{2},\frac{2s'+t'(t'-1)}{2}}^l \sum_{i=1}^{n-1}\sum_{j=i+1}^n(U_{is'}U_{jt'} + U_{js'}U_{it'})\;G_{l,ij} \nonumber \\
	&=& \sum_{i=1}^{n-1}\sum_{j=i+1}^n\left\{G_{l,ij}\sum_{1\leq s\leq t \leq K}x_{\frac{2s+t(t-1)}{2}} \sum_{1\leq s'\leq t' \leq K} \tilde{\Sigma}_{\frac{2s+t(t-1)}{2},\frac{2s'+t'(t'-1)}{2}}^l (U_{is'}U_{jt'} + U_{js'}U_{it'})\right\}\nonumber \\
	&:=& \sum_{i=1}^{n-1}\sum_{j=i+1}^n\bar{F}_{l,ij},
\end{eqnarray}
where we let $\bar{F}_{l,ij}$ denote the portion enclosed in braces in the second equality, and the first equality follows from (\ref{UGlU_st}). Given $l$,  it is obvious that the $\bar{F}_{l,ij}$ terms are zero-mean and independent of each other for all $1\leq i<j \leq n$. In accordance with Assumption \ref{cov eigenvalue}, we have
\begin{eqnarray*}
	|\bar{F}_{l,ij}| &\leq & \frac{c'}{n} \sum_{1\leq s\leq t \leq K}x_{\frac{2s+t(t-1)}{2}} \sum_{1\leq s'\leq t' \leq K} \left|\tilde{\Sigma}_{\frac{2s+t(t-1)}{2},\frac{2s'+t'(t'-1)}{2}}^l \right|  \nonumber \\
	&\leq &\frac{c'}{n}\frac{K(K+1)}{2}\|x\|_2 \left\{\sum_{\substack{1\leq s\leq t \leq K \\[3pt] 1\leq s'\leq t' \leq K}}\left(\tilde{\Sigma}_{\frac{2s+t(t-1)}{2},\frac{2s'+t'(t'-1)}{2}}^l \right)^2 \right\}^{1/2}\nonumber \\
	&\leq & \frac{c'}{n}K^3\|x\|_2\lambda_{\max}(\widetilde{\Sigma}^l)\nonumber \\
	&\leq & \frac{c'K^3\|x\|_2}{n\sqrt{\lambda_{\min}(\Sigma^l)}} = o(\|x_2\|_2)
\end{eqnarray*}
for all $l \in [L]$ and $1\leq i<j\leq n$. To affirm the Lindeberg condition, we proceed by calculating the sum of ${\rm Var}(\bar{F}_{l,ij})$ for all $1\leq i<j \leq n$. By a simple calculation, we have ${\rm Var}\left(x^T(\Sigma^l)^{-1/2}{\rm vec}(U^TG_lU)\right) = \|x\|_2^2$. Therefore, for any $\varepsilon > 0$ and sufficiently large $n$, it holds true that $|\bar{F}_{l,ij}| < \varepsilon \left[{\rm Var}\left(x^T(\Sigma^l)^{-1/2}{\rm vec}(U^TG_lU)\right)\right]^{1/2}$. This subsequently satisfies the Lindeberg condition as given by
\begin{equation*}\label{lindeberg_vec}
	\lim\limits_{n\rightarrow \infty}\frac{1}{\|x\|_2^2}\sum_{i=1}^{n-1}\sum_{j=i+1}^n \mathbb E \left(|\bar{F}_{l,ij}|^2 \mathbb I_{\{ |\bar{F}_{l,ij}| > \varepsilon[{\rm Var}(x^T(\Sigma^l)^{-1/2}{\rm vec}(U^TG_lU))]^{1/2}\}}\right) = 0.
\end{equation*}
We should also emphasize that $U$, $G_l$, and $\bar{F}_{l,ij}$ are all dependent on $n$. For the sake of simplicity, we have not explicitly displayed the dependence on $n$ in our notation. By the Lindeberg-Feller Central Limit Theorem and the Cram\'er-Wold theorem (Thoerem 3.9.5 of \cite{durrett2010probability}), we have 
\begin{equation*}
	\left(\Sigma^l\right)^{-1/2} {\rm vec}\left (U^TG_lU\right) \rightarrow_d \mathcal N (\bm 0, \bm I)
\end{equation*}
for each $l\in[L]$. Here, $\mathcal N (\bm 0, \bm I)$ is a $K(K+1)/2$ dimensional standard normal distribution. The claim follows by combining the above discussion with (\ref{Transform}). $\hfill\qedsymbol$

\subsubsection*{Proof of Theorem \ref{clt_dc}}
Consistent with the proof of Theorem \ref{clt}, a random matrix $E_l$ exists such that \ref{Transform} holds. According to Theorem \ref{eigenspace bound dc}, the bias term $E_l$ satisfies 
\begin{eqnarray*}
	\|ZE_lZ^T\|_F &\leq& c(\|Q_l\|_2 + \sqrt{\|Q_l\|_{1, \infty}}) \left( \frac{n^2}{\|\psi\|_2^8} + \frac{n^3\log(L+n)}{L\rho\|\psi\|_2^8} \right) \nonumber\\
	& & ~+ c\sqrt{\|Q_l\|_{1, \infty}}\left(\frac{n}{\|\psi\|_2^4} + \frac{n^{3/2}\log^{1/2}(L+n)}{L^{1/2}\rho^{1/2 }\|\psi\|_2^4} \right) + \sqrt{K}\rho
\end{eqnarray*}
with high probability. Furthermore, given that every element of $Q_l$ is at most $\rho$, neither $\|Q_l\|_2$ nor $\|Q_l\|_{1,\infty}$ exceeds $n\rho$. Thus, we can deduce that $\|ZE_lZ^T\|_F \leq c\left(\frac{n^2\log^{1/2}(L+n)}{L^{1/2}\|\psi\|_2^4}  + \max\{\frac{n^{3/2}\rho^{1/2}}{\|\psi\|_2^4}, \rho\} \right)$.
Under certain setting for $\psi$, $ZE_lZ^T$ tends to vanish as the number of layers $L$ increase, and as the overall edge density $\rho$ decrease.

To demonstrate the central limit theorem, it suffices to ascertain the asymptotic normality of $U^TG_lU$. For part (a), recall that the decomposition of $(UG_lU^T)_{s, t}$ in \ref{UGlU_st} for all $1 \leq s \leq t \leq K$. By the definition of $U$, that is $\tilde\Theta$ here, and according to Assumption \ref{bal_dc}, we have $F_{l,ij} \leq c/\|\psi\|_2^2$ for all $l \in [L]$ and $1 \leq i \leq j \leq n$. Combing with Assumption \ref{var_dc}, the Lindeberg condition in (\ref{lindeberg}) holds. For part (b), recall that the decomposition of $x^T(\Sigma^l)^{-1/2}{\rm vec}(U^TG_lU)$  and the definition of $\bar{F}_{l,ij}$ in (\ref{ver_UGlU}). Under Assumption \ref{cov eigenvalue dc}, it follows that $|\bar{F}_{l,ij}| \leq \frac{c \|x\|_2}{\|\psi\|_2^2\sqrt{\lambda_{\min}(\Sigma^l)}} = o(\|x\|_2)$. The rest proof is similar to that of Theorem \ref{clt}, we omit it here. $\hfill\qedsymbol$

\subsection{Auxiliary lemmas}\label{AuxiliaryLemmas}
\begin{lemma}[Theorem 5.2 of \cite{lei2015consistency}]\label{Spectral bound}
Let $A$ be the adjacency matrix of a random graph on $n$ nodes in which edges occur independently. Set $\mathbb E(A) = P = (p_{ij})$ for ${i,j=1,\ldots,n}$, and assume that $n \max_{i,j} p_{ij} < d$ for $d > c \log n$ and $c > 0$. Then, for any $r > 0$ there exists a constant $C = C(r, c)$ such that
\begin{equation*}
	\| A - P \|_2 \leq C \sqrt{d}
\end{equation*}
with probability at least $1 - n^{-r}$.
\end{lemma}

\subsection{Additional results for data analysis}\label{additional real data}
\begin{table}[h]
    \caption{Categories of cereals as defined by the FAO, with `n.e.c.' representing `not elsewhere classified'. Cereal numbers correspond to layers in the WFAT network depicted in Figure \ref{testing for real data trade}.}
\centering
\begin{tabular}[]{|c|}
\hline
\textbf{1}.Malt, whether or not roasted;\;
\textbf{2}.Rice;\;
\textbf{3}.Canary seed;\;
\textbf{4}.Buckwheat;\; 
\textbf{5}.Germ of \\ maize;\;
\textbf{6}.Triticale;\;
\textbf{7}.Cereals n.e.c.;\;
\textbf{8}.Millet;\;
\textbf{9}.Sorghum;\;
\textbf{10}.Bran of maize;\;
\textbf{11}.Bran \\ of cereals n.e.c.;\;
\textbf{12}.Flour of mixed grain;\;
\textbf{13}.Flour of rice;\;
\textbf{14}.Rye;\;
\textbf{15}.Malt extract; \\
\textbf{16}.Rice, milled (husked);\;
\textbf{17}.Oats;\;
\textbf{18}.Cereal preparations;\;
\textbf{19}.Oats, rolled;\;
\textbf{20}.Bran of \\ wheat;\;
\textbf{21}.Flour of cereals n.e.c.;\;
\textbf{22}.Rice, broken;\;
\textbf{23}.Flour of maize;\;
\textbf{24}.Gluten feed \\ and meal;\;
\textbf{25}.Bread;\;
\textbf{26}.Husked rice;\;
\textbf{27}.Wheat and meslin flour;\;
\textbf{28}.Communion \\ wafers, empty cachets of a kind suitable for pharmaceutical use, sealing wafers, rice \\ paper and similar products;\;
\textbf{29}.Barley;\;
\textbf{30}.Breakfast cereals;\;
\textbf{31}.Uncooked pasta, not \\ stuffed or otherwise prepared;\;
\textbf{32}.Maize (corn);\;
\textbf{33}.Wheat;\;
\textbf{34}.Mixes and doughs for \\ the preparation of bakers' wares;\;
\textbf{35}.Food preparations of flour, meal or malt extract;\\
\textbf{36}.Rice, milled;\;
\textbf{37}.Pastry \\
\hline
\end{tabular}
    \label{items}
\end{table}

\begin{figure}[!htbp]
    \centering
    \includegraphics[width=0.45\textwidth]{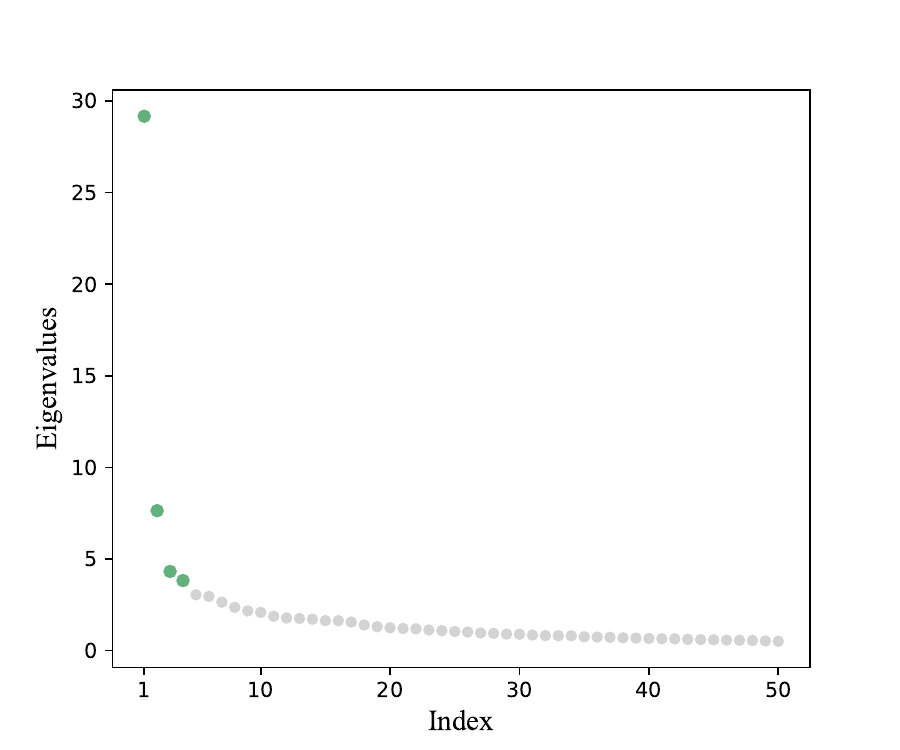}
    \caption{Scree plot of the top eigenvalues of $\sum_{l=1}^L A_l^2 / n$ of the WFAT data graph. The plot show the 50 largest eigenvalues ordered by magnitude.}
    \label{scree plot}
\end{figure}	
\noindent This section presents the additional experimental results that are not shown in the main text. In Section \ref{Real data analysis}, for the WFAT dataset, we initially chose the number of communities as $K = 5$ in the multi-layer SBM analysis, corresponding naturally to the geographical division into five continents. Here, we adopt an alternative approach to determine $K$. By observing an elbow in the scree plot of the top absolute eigenvalues of the sum of squared adjacency matrices at the 4th position, as shown in Figure \ref{scree plot}, we select $K = 4$. Figure \ref{WFAT data K = 4} presents the outcomes of the multiple comparisons using a Holm type step-down procedure, where olive green indicates the acceptance and white denotes the rejection of the individual hypotheses. We control the Holm procedure to ensure that the family-wise error rate is no greater than $\alpha = 0.05$. In this setup, layer $l$ corresponds to the cereal numbered $l$ in Table \ref{items}. Similar to Figure \ref{mtp_holm}, the outcomes of this multiple comparison procedure also exhibit several distinct blocks where all individual hypotheses are simultaneously accepted within each block. Moreover, the results are largely consistent with those observed when $K = 5$.  This strongly suggests that some global cereal trade patterns are similar and also demonstrates the stability of our method.

\begin{figure}[!h]
    \centering
    	\includegraphics[width=0.5\textwidth]{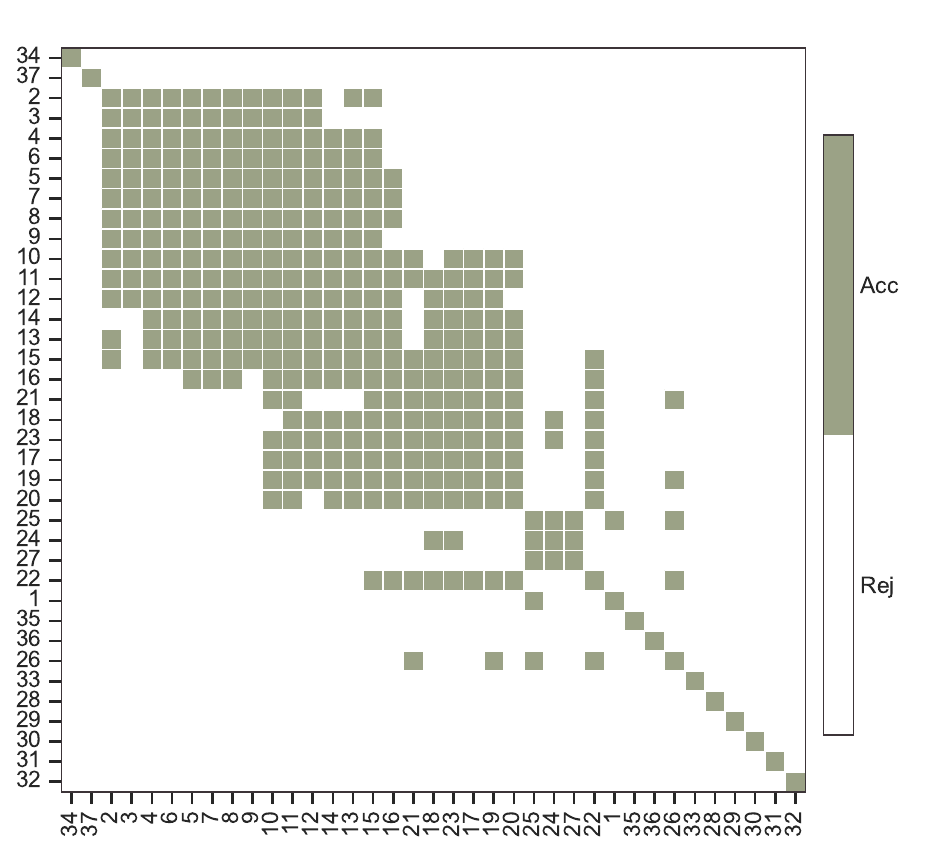}
    \caption{Results of the multiple comparisons conducted using a Holm type step-down procedure for the WFAT data, with the number of  communities set to $K = 4$. Olive green indicates the acceptance of the individual hypotheses. Each layer $l$ corresponding to the cereal numbered $l$ in Table \ref{items}.}
    \label{WFAT data K = 4}
\end{figure}

\newpage
\bibliographystyle{chicago}
\bibliography{reference}

\begin{thebibliography}{}

\bibitem[\protect\citeauthoryear{Agterberg and Cape}{Agterberg and
  Cape}{2023}]{agterberg2023overview}
Agterberg, J. and J.~Cape (2023).
\newblock An overview of asymptotic normality in stochastic blockmodels:
  Cluster analysis and inference.
\newblock {\em arXiv preprint arXiv:2305.06353\/}.

\bibitem[\protect\citeauthoryear{Agterberg, Lubberts, and Arroyo}{Agterberg
  et~al.}{2022}]{agterberg2022joint}
Agterberg, J., Z.~Lubberts, and J.~Arroyo (2022).
\newblock Joint spectral clustering in multilayer degree-corrected stochastic
  blockmodels.
\newblock {\em arXiv preprint arXiv:2212.05053\/}.

\bibitem[\protect\citeauthoryear{Airoldi, Blei, Fienberg, and Xing}{Airoldi
  et~al.}{2008}]{airoldi2008mixed}
Airoldi, E.~M., D.~Blei, S.~Fienberg, and E.~Xing (2008).
\newblock Mixed membership stochastic blockmodels.
\newblock {\em Journal of Machine Learning Research\/}~{\em 9\/}(65),
  1981--2014.

\bibitem[\protect\citeauthoryear{Arroyo, Athreya, Cape, Chen, Priebe, and
  Vogelstein}{Arroyo et~al.}{2021}]{arroyo2021inference}
Arroyo, J., A.~Athreya, J.~Cape, G.~Chen, C.~E. Priebe, and J.~T. Vogelstein
  (2021).
\newblock Inference for multiple heterogeneous networks with a common invariant
  subspace.
\newblock {\em Journal of Machine Learning Research\/}~{\em 22\/}(142), 1--49.

\bibitem[\protect\citeauthoryear{Bhatia}{Bhatia}{1997}]{Bhatia1997matrix}
Bhatia, R. (1997).
\newblock {\em Matrix analysis}.
\newblock Springer-Verlag, New York.

\bibitem[\protect\citeauthoryear{Bhattacharyya and Chatterjee}{Bhattacharyya
  and Chatterjee}{2018}]{bhattacharyya2018spectral}
Bhattacharyya, S. and S.~Chatterjee (2018).
\newblock Spectral clustering for multiple sparse networks: I.
\newblock {\em arXiv preprint arXiv:1805.10594\/}.

\bibitem[\protect\citeauthoryear{Bickel, Choi, Chang, and Zhang}{Bickel
  et~al.}{2013}]{bickel2013asymptotic}
Bickel, P., D.~Choi, X.~Chang, and H.~Zhang (2013).
\newblock Asymptotic normality of maximum likelihood and its variational
  approximation for stochastic blockmodels1.
\newblock {\em The Annals of Statistics\/}~{\em 41\/}(4), 1922--1943.

\bibitem[\protect\citeauthoryear{De~Domenico, Nicosia, Arenas, and
  Latora}{De~Domenico et~al.}{2015}]{de2015structural}
De~Domenico, M., V.~Nicosia, A.~Arenas, and V.~Latora (2015).
\newblock Structural reducibility of multilayer networks.
\newblock {\em Nature Communications\/}~{\em 6}, 6864.

\bibitem[\protect\citeauthoryear{Dudoit, Van Der~Laan, and van~der Laan}{Dudoit
  et~al.}{2008}]{dudoit2008multiple}
Dudoit, S., M.~J. Van Der~Laan, and M.~J. van~der Laan (2008).
\newblock {\em Multiple testing procedures with applications to genomics}.
\newblock Springer, New York.

\bibitem[\protect\citeauthoryear{Durrett}{Durrett}{2010}]{durrett2010probability}
Durrett, R. (2010).
\newblock {\em Probability: theory and examples}.
\newblock Cambridge University Press, New York.

\bibitem[\protect\citeauthoryear{Fan, Fan, Han, and Lv}{Fan
  et~al.}{2022a}]{fan2022asymptotic}
Fan, J., Y.~Fan, X.~Han, and J.~Lv (2022a).
\newblock Asymptotic theory of eigenvectors for random matrices with diverging
  spikes.
\newblock {\em Journal of the American Statistical Association\/}~{\em
  117\/}(538), 996--1009.

\bibitem[\protect\citeauthoryear{Fan, Fan, Han, and Lv}{Fan
  et~al.}{2022b}]{fan2022simple}
Fan, J., Y.~Fan, X.~Han, and J.~Lv (2022b).
\newblock Simple: Statistical inference on membership profiles in large
  networks.
\newblock {\em Journal of the Royal Statistical Society Series B: Statistical
  Methodology\/}~{\em 84\/}(2), 630--653.

\bibitem[\protect\citeauthoryear{Fan, Wang, Wang, and Zhu}{Fan
  et~al.}{2019}]{fan2019distributed}
Fan, J., D.~Wang, K.~Wang, and Z.~Zhu (2019).
\newblock Distributed estimation of principal eigenspaces.
\newblock {\em The Annals of Statistics\/}~{\em 47\/}(6), 3009--3031.

\bibitem[\protect\citeauthoryear{Han, Xu, and Airoldi}{Han
  et~al.}{2015}]{han2015consistent}
Han, Q., K.~Xu, and E.~Airoldi (2015).
\newblock Consistent estimation of dynamic and multi-layer block models.
\newblock In {\em International Conference on Machine Learning}, pp.\
  1511--1520. PMLR.

\bibitem[\protect\citeauthoryear{He, Knowles, and Marcozzi}{He
  et~al.}{2019}]{he2019local}
He, Y., A.~Knowles, and M.~Marcozzi (2019).
\newblock {Local law and complete eigenvector delocalization for supercritical
  Erd{\H o}s--R{\'e}nyi graphs}.
\newblock {\em The Annals of Probability\/}~{\em 47\/}(5), 3278 -- 3302.

\bibitem[\protect\citeauthoryear{Holland, Laskey, and Leinhardt}{Holland
  et~al.}{1983}]{holland1983stochastic}
Holland, P.~W., K.~B. Laskey, and S.~Leinhardt (1983).
\newblock Stochastic blockmodels: First steps.
\newblock {\em Social Networks\/}~{\em 5\/}(2), 109--137.

\bibitem[\protect\citeauthoryear{Holme and Saram{\"a}ki}{Holme and
  Saram{\"a}ki}{2012}]{holme2012temporal}
Holme, P. and J.~Saram{\"a}ki (2012).
\newblock Temporal networks.
\newblock {\em Physics Reports\/}~{\em 519\/}(3), 97--125.

\bibitem[\protect\citeauthoryear{Jin, Ke, Luo, and Wang}{Jin
  et~al.}{2023}]{jin2022optimal}
Jin, J., Z.~T. Ke, S.~Luo, and M.~Wang (2023).
\newblock Optimal estimation of the number of network communities.
\newblock {\em Journal of the American Statistical Association\/}~{\em
  118\/}(543), 2101--2116.

\bibitem[\protect\citeauthoryear{Jing, Li, Lyu, and Xia}{Jing
  et~al.}{2021}]{jing2021community}
Jing, B.-Y., T.~Li, Z.~Lyu, and D.~Xia (2021).
\newblock Community detection on mixture multilayer networks via regularized
  tensor decomposition.
\newblock {\em The Annals of Statistics\/}~{\em 49\/}(6), 3181--3205.

\bibitem[\protect\citeauthoryear{Karrer and Newman}{Karrer and
  Newman}{2011}]{karrer2011stochastic}
Karrer, B. and M.~E. Newman (2011).
\newblock Stochastic blockmodels and community structure in networks.
\newblock {\em Physical Review E\/}~{\em 83\/}(1), 016107.

\bibitem[\protect\citeauthoryear{Kivel{\"a}, Arenas, Barthelemy, Gleeson,
  Moreno, and Porter}{Kivel{\"a} et~al.}{2014}]{kivela2014multilayer}
Kivel{\"a}, M., A.~Arenas, M.~Barthelemy, J.~P. Gleeson, Y.~Moreno, and M.~A.
  Porter (2014).
\newblock Multilayer networks.
\newblock {\em Journal of Complex Networks\/}~{\em 2\/}(3), 203--271.

\bibitem[\protect\citeauthoryear{Lehmann and Romano}{Lehmann and
  Romano}{2005}]{lehmann2005testing}
Lehmann, E.~L. and J.~P. Romano (2005).
\newblock {\em Testing statistical hypotheses}.
\newblock Springer, New York.

\bibitem[\protect\citeauthoryear{Lei and Lin}{Lei and Lin}{2023}]{LeiJ2022Bias}
Lei, J. and K.~Z. Lin (2023).
\newblock Bias-adjusted spectral clustering in multi-layer stochastic block
  models.
\newblock {\em Journal of the American Statistical Association\/}~{\em
  118\/}(544), 2433--2445.

\bibitem[\protect\citeauthoryear{Lei and Rinaldo}{Lei and
  Rinaldo}{2015}]{lei2015consistency}
Lei, J. and A.~Rinaldo (2015).
\newblock Consistency of spectral clustering in stochastic block models.
\newblock {\em The Annals of Statistics\/}~{\em 43\/}(1), 215--237.

\bibitem[\protect\citeauthoryear{Mucha, Richardson, Macon, Porter, and
  Onnela}{Mucha et~al.}{2010}]{mucha2010community}
Mucha, P.~J., T.~Richardson, K.~Macon, M.~A. Porter, and J.-P. Onnela (2010).
\newblock Community structure in time-dependent, multiscale, and multiplex
  networks.
\newblock {\em Science\/}~{\em 328\/}(5980), 876--878.

\bibitem[\protect\citeauthoryear{Noble}{Noble}{2009}]{noble2009does}
Noble, W.~S. (2009).
\newblock How does multiple testing correction work?
\newblock {\em Nature Biotechnology\/}~{\em 27\/}(12), 1135--1137.

\bibitem[\protect\citeauthoryear{Noroozi and Pensky}{Noroozi and
  Pensky}{2024}]{noroozi2022sparse}
Noroozi, M. and M.~Pensky (2024).
\newblock Sparse subspace clustering in diverse multiplex network model.
\newblock {\em Journal of Multivariate Analysis\/}~{\em 203}, 105333.

\bibitem[\protect\citeauthoryear{Paul and Chen}{Paul and
  Chen}{2016}]{paul2016consistent}
Paul, S. and Y.~Chen (2016).
\newblock Consistent community detection in multi-relational data through
  restricted multi-layer stochastic blockmodel.
\newblock {\em Electronic Journal of Statistics\/}~{\em 10\/}(2), 3807--3870.

\bibitem[\protect\citeauthoryear{Paul and Chen}{Paul and
  Chen}{2020}]{paul2020spectral}
Paul, S. and Y.~Chen (2020).
\newblock Spectral and matrix factorization methods for consistent community
  detection in multi-layer networks.
\newblock {\em The Annals of Statistics\/}~{\em 48\/}(1), 230--250.

\bibitem[\protect\citeauthoryear{Rubin-Delanchy, Cape, Tang, and
  Priebe}{Rubin-Delanchy et~al.}{2022}]{rubin2022statistical}
Rubin-Delanchy, P., J.~Cape, M.~Tang, and C.~E. Priebe (2022).
\newblock A statistical interpretation of spectral embedding: The generalised
  random dot product graph.
\newblock {\em Journal of the Royal Statistical Society Series B: Statistical
  Methodology\/}~{\em 84\/}(4), 1446--1473.

\bibitem[\protect\citeauthoryear{Rudelson and Vershynin}{Rudelson and
  Vershynin}{2015}]{rudelson2015Delocalization}
Rudelson, M. and R.~Vershynin (2015).
\newblock {Delocalization of eigenvectors of random matrices with independent
  entries}.
\newblock {\em Duke Mathematical Journal\/}~{\em 164\/}(13), 2507--2538.

\bibitem[\protect\citeauthoryear{Su, Wang, and Zhang}{Su
  et~al.}{2020}]{su2020strong}
Su, L., W.~Wang, and Y.~Zhang (2020).
\newblock Strong consistency of spectral clustering for stochastic block
  models.
\newblock {\em IEEE Transactions on Information Theory\/}~{\em 66\/}(1),
  324--338.

\bibitem[\protect\citeauthoryear{Su, Guo, Chang, and Yang}{Su
  et~al.}{2024}]{su2023}
Su, W., X.~Guo, X.~Chang, and Y.~Yang (2024).
\newblock Spectral co-clustering in multi-layer directed networks.
\newblock {\em Computational Statistics \& Data Analysis\/}~{\em 198}, 107987.

\bibitem[\protect\citeauthoryear{Tang, Cape, and Priebe}{Tang
  et~al.}{2022}]{tang2022asymptotically}
Tang, M., J.~Cape, and C.~E. Priebe (2022).
\newblock Asymptotically efficient estimators for stochastic blockmodels: The
  naive mle, the rank-constrained mle, and the spectral estimator.
\newblock {\em Bernoulli\/}~{\em 28\/}(2), 1049--1073.

\bibitem[\protect\citeauthoryear{Tang and Priebe}{Tang and
  Priebe}{2018}]{tang2018limit}
Tang, M. and C.~E. Priebe (2018).
\newblock Limit theorems for eigenvectors of the normalized laplacian for
  random graphs.
\newblock {\em The Annals of Statistics\/}~{\em 46\/}(5), 2360--2415.

\bibitem[\protect\citeauthoryear{Valles-Catala, Massucci, Guimera, and
  Sales-Pardo}{Valles-Catala et~al.}{2016}]{valles2016multilayer}
Valles-Catala, T., F.~A. Massucci, R.~Guimera, and M.~Sales-Pardo (2016).
\newblock Multilayer stochastic block models reveal the multilayer structure of
  complex networks.
\newblock {\em Physical Review X\/}~{\em 6\/}(1), 011036.

\bibitem[\protect\citeauthoryear{Vu and Lei}{Vu and Lei}{2013}]{vu2013minimax}
Vu, V.~Q. and J.~Lei (2013).
\newblock Minimax sparse principal subspace estimation in high dimensions.
\newblock {\em The Annals of Statistics\/}~{\em 41\/}(6), 2905--2947.

\bibitem[\protect\citeauthoryear{Young and Scheinerman}{Young and
  Scheinerman}{2007}]{young2007random}
Young, S.~J. and E.~R. Scheinerman (2007).
\newblock Random dot product graph models for social networks.
\newblock In {\em International Workshop on Algorithms and Models for the
  Web-Graph}, pp.\  138--149. Springer.

\bibitem[\protect\citeauthoryear{Zhang, Guo, and Chang}{Zhang
  et~al.}{2022}]{zhang2022randomized}
Zhang, H., X.~Guo, and X.~Chang (2022).
\newblock Randomized spectral clustering in large-scale stochastic block
  models.
\newblock {\em Journal of Computational and Graphical Statistics\/}~{\em
  31\/}(3), 887--906.

\bibitem[\protect\citeauthoryear{Zheng and Tang}{Zheng and
  Tang}{2022}]{zheng2022limit}
Zheng, R. and M.~Tang (2022).
\newblock Limit results for distributed estimation of invariant subspaces in
  multiple networks inference and pca.
\newblock {\em arXiv preprint arXiv:2206.04306\/}.

\end{thebibliography}
\end{document}